\documentclass[12pt]{article}
\usepackage{geometry} % see geometry.pdf on how to lay out the page. There's lots.
\usepackage{amssymb}
\usepackage{amsmath}
\usepackage{mathrsfs}
\usepackage{amsthm}
\newtheorem{thm}{\normalfont{\textsc{Theorem}}}
\newtheorem{prop}[thm]{\normalfont{\textsc{Proposition}}}
\newtheorem{lemma}[thm]{\normalfont{\textsc{Lemma}}}

\theoremstyle{remark}
\newtheorem{remark}[thm]{\textsc{Remark}}
\newtheorem{defn}[thm]{\textsc{Definition}}
\newtheorem{question}[thm]{\textsc{Question}}

\newenvironment{altproof3}{\vspace{1ex}\noindent{\textit{Proof of Proposition~\ref{unibound}}}\hspace{0.5em}}
	{\hfill\qed\vspace{1ex}}
	\newenvironment{altproof4}{\vspace{1ex}\noindent{\textit{Proof of Lemma~\ref{frount}}}\hspace{0.5em}}
	{\hfill\qed\vspace{1ex}}
 
\def\cprime{$'$}

\newcommand{\nr}[1]{\vspace{0.1ex}\noindent\hspace*{9mm}\llap{\textup{(#1)}}}
 
\newcommand{\norm}[1]{\left\Vert#1\right\Vert}
\newcommand{\abs}[1]{\left\vert#1\right\vert}
\newcommand{\set}[1]{\left\{#1\right\}}
\newcommand{\field}{\mathbb K}
\newcommand{\nat}{\mathbb N}
\newcommand{\real}{\mathbb R}
\newcommand{\complex}{\mathbb C}
\newcommand{\eps}{\varepsilon}
\newcommand{\To}{\longrightarrow}
\newcommand{\functional}[2]{\langle #1 ,\,  #2 \rangle }
\newcommand{\cball}[1]{B_{#1}}
\newcommand{\dual}{\sp{\ast}}
\newcommand{\compactop}{\ensuremath{\mathscr{K}}}
\newcommand{\szlenkop}[1]{\ensuremath{\mathscr{S\hspace{-0.8mm}Z}_{\hspace{-1mm}#1}}}
\newcommand{\allop}{\ensuremath{\mathscr{B}}}
\newcommand{\opideal}{\ensuremath{\mathscr{I}}}
\newcommand{\mszlenk}[1]{{\rm Sz}(#1)}
\newcommand{\meeszlenk}[2]{{\rm Sz}_{#1} ( #2 )}
\newcommand{\mdset}[3]{s^{#1}_{#2}( #3 )}
\newcommand{\diam}{{\rm diam}}
\newcommand{\ord}{{\textsc{Ord}}}
\newcommand{\cf}[1]{cf \left( #1\right)}
\newcommand{\fin}{^{ < \, \infty}}
\newcommand{\eff}{\mathcal{F}}
\newcommand{\gee}{\mathcal{G}}
\newcommand{\arr}{\mathcal{R}}
\newcommand{\fresh}{Fr{\'e}chet}
\newcommand{\hajek}{H{\'a}jek}
\newcommand{\kreinmilman}{{K}re\u\i n-Mil{\cprime}man}
\newcommand{\pelch}{Pe{\l}czy{\'n}ski}
\newcommand{\reinov}{Re{\u\i}nov}
\newcommand{\cyril}{Tsirel$'$son}
\numberwithin{thm}{section}
\numberwithin{equation}{section}

\begin{document}
\title{Direct sums and the {S}zlenk index\footnote{Research supported by an ANU PhD Scholarship. The present work forms part of the author's doctoral dissertation, written at the Australian National University under the supervision of Dr. Richard J. Loy.}}
\author{Philip A. H. Brooker}
\date{} % delete this line to display the current date 
\maketitle

\begin{abstract}
For $\alpha$ an ordinal and $1<p<\infty$, we determine a necessary and sufficient condition for an $\ell_p$-direct sum of operators to have Szlenk index not exceeding $\omega^\alpha$. It follows from our results that the Szlenk index of an $\ell_p$-direct sum of operators is determined in a natural way by the behaviour of the $\eps$-Szlenk indices of its summands. Our methods give similar results for $c_0$-direct sums.
\end{abstract}

\section*{Introduction}
The Szlenk index was introduced by W.~Szlenk in his influential paper \cite{Szlenk1968}, where an ordinal index was used to show that the class of all separable, reflexive Banach spaces contains no universal element. Since then, the Szlenk index and its variants have taken on an increasingly important role in the study of Banach spaces and their operators. We refer the reader to the surveys \cite{Lancien2006} and \cite{Rosenthal2003} for details on some of the main applications of the Szlenk index.

A class of closed operator ideals naturally related to the Szlenk index has been introduced and systematically studied by the present author in \cite{Brookerc}. These operator ideals are denoted $\szlenkop{\alpha}$, where $\alpha$ is an ordinal, and elements of $\szlenkop{\alpha}$ are known as \emph{$\alpha$-Szlenk operators}. The operator ideals $\szlenkop{\alpha}$ are studied in \cite{Brookerc} with regard to their operator ideal properties and their relationship to other closed operator ideals, in particular the class of Asplund operators.

The purpose of the present paper is to present a detailed analysis of the behaviour of the Szlenk index under the process of taking $c_0$ and $\ell_p$-direct sums of operators. In particular, we give a precise formulation of the Szlenk index of a direct sum of operators in terms of the behaviour of the $\eps$-Szlenk indices of the summands. Our motivation for this is as follows. Firstly, forming direct sums is a fundamental construction in Banach space theory, often being used to construct examples with a particular property, and so we feel it essential to understand precisely how the Szlenk index behaves under this procedure. Secondly, we are motivated by the following basic question of operator ideal theory: \begin{question}\label{redpenquest} Let $\opideal$ be a given operator ideal. Does $\opideal$ have the factorisation property? That is, does every element of $\opideal$ factor continuously and linearly through a Banach space whose identity operator belongs to $\opideal$? \end{question}
In \cite{Brookerc}, results and techniques developed in the current paper are applied to obtain both positive and negative answers to Question~\ref{redpenquest} for the case $\opideal = \szlenkop{\alpha}$, with the answer depending upon ordinal properties of $\alpha$.

We now outline the structure of the current paper. In Section~\ref{prelimlim} we detail necessary notation and background results regarding the Szlenk index, including several relevant results from \cite{Brookerc}. Our main results are presented in Section~\ref{mainstuffing}. Firstly, we consider the Szlenk index of $\ell_1$ and $\ell_\infty$-direct sums; this case is rather straightforward, but worth noting explicitly for the sake of completeness. We then move on to our main concern, providing a formulation of the Szlenk index of $c_0$ and $\ell_p$-direct sums of operators, where $1<p<\infty$ (see, in particular, Theorem~\ref{punibound}). This case is far more subtle than the case of $\ell_1$ and $\ell_\infty$-direct sums and, as such, requires substantially more effort to accomplish the desired formulation of the Szlenk index of the direct sum. Section~\ref{mainstuffing} concludes with some applications of the earlier operator theoretic results to the Szlenk index of Banach spaces. The final section, Section~\ref{techsect}, constitutes almost half of the paper and is devoted to proving the main technical lemma used in Section~\ref{mainstuffing}, namely Lemma~\ref{frount}.

\section{Preliminaries}\label{prelimlim}
Banach spaces are typically denoted by the letters $E$ and $F$. For a Banach space $E$ and nonempty bounded $S\subseteq E$, we define $\abs{S} := \sup\{\norm{x}\mid x\in S\}$. By $\cball{E}$ we denote the closed unit ball of $E$, and by $I_E$ the identity operator of $E$. The class of all bounded linear operators between arbitrary Banach spaces is denoted by $\allop$, and the class of all compact operators by $\compactop$. We write $\ord$ for the class of all ordinals, whose elements are typically denoted by the lower-case Greek letters $\alpha$, $\beta$ and $\gamma$. For $\Lambda$ a set, $\Lambda\fin$ denotes the set of all nonempty finite subsets of $\Lambda$. When $\Lambda$ denotes the index set over which we take a direct sum or direct product, it is always assumed that $\Lambda$ is nonempty.

Let $p\in\set{0}\cup(1, \, \infty)$ and $q\in[1,\,\infty)$. We say that \emph{$q$ is dual to $p$}, or equivalently, \emph{$p$ is predual to $q$}, if $(p, \, q)\in \set{(0,\,1)}\cup\set{(r,\,r(r-1)^{-1})\mid r\in(1,\,\infty)}$.

For $1\leqslant p\leqslant \infty$, a set $\Lambda$ and Banach spaces $E_\lambda$, $\lambda \in \Lambda$, the $\ell_p$-direct sum of $\set{E_\lambda \mid \lambda \in \Lambda}$ is denoted $(\bigoplus_{\lambda \in \Lambda}E_\lambda)_p$, and the $c_0$-direct sum of $\set{E_\lambda \mid \lambda \in \Lambda}$ is denoted $(\bigoplus_{\lambda \in \Lambda}E_\lambda)_0$.
If there is a Banach space $E$ such that $E_\lambda = E$ for all $\lambda \in \Lambda$, then we may also write the $\ell_p$-direct sum and the $c_0$-direct sum as $\ell_p (\Lambda, \, E)$ and $c_0 (\Lambda, \, E)$, respectively. Throughout, for $1<p, q<\infty$ satisfying $p+q=pq$, we implicitly identify $(\bigoplus_{\lambda \in \Lambda}E_\lambda)\dual_p$ with $(\bigoplus_{\lambda \in \Lambda}E_\lambda\dual)_q$, so that the dual of a direct sum is the dual direct sum of the duals of the spaces $E_\lambda$. Making this identification allows us to consider direct products of the form $\prod_{\lambda \in \Lambda}K_\lambda$, where $K_\lambda \subseteq E_\lambda\dual$ and $(\abs{K_\lambda})_{\lambda \in \Lambda} \in \ell_q (\Lambda)$, as subsets of $(\bigoplus_{\lambda \in \Lambda}E_\lambda)\dual_p$. Similarly, $(\bigoplus_{\lambda \in \Lambda}E_\lambda)\dual_0$ is naturally identified with $(\bigoplus_{\lambda \in \Lambda}E_\lambda\dual)_1$ throughout.

Let $\Lambda$ be a set, $\set{E_\lambda \mid \lambda \in \Lambda}$ a family of Banach spaces indexed by $\Lambda$ and $p=0$ or $1<p<\infty$. For $\arr \subseteq \Lambda$, we denote by $U_\arr$ the canonical injection of $(\bigoplus_{\lambda \in \arr}E_\lambda)_p $ into $(\bigoplus_{\lambda \in \Lambda}E_\lambda)_p$, and by $P_\arr$ the canonical surjection of $(\bigoplus_{\lambda \in \Lambda}E_\lambda)_p$ onto $(\bigoplus_{\lambda \in \arr}E_\lambda)_p$. %Let $\Lambda$ be a set, $\set{E_\lambda \mid \lambda \in \Lambda}$ a family of Banach spaces and $p=0$ or $1<p<\infty$. For $\arr \subseteq \Lambda$, $U_\arr$ denotes the canonical injection of $(\bigoplus_{\lambda \in \arr}E_\lambda)_p $ into $(\bigoplus_{\lambda \in \Lambda}E_\lambda)_p$, and $P_\arr$ the canonical surjection of $(\bigoplus_{\lambda \in \Lambda}E_\lambda)_p$ onto $(\bigoplus_{\lambda \in \arr}E_\lambda)_p$. We shall at times consider operators $T: (\bigoplus_{\lambda \in \Lambda}E_\lambda)_p  \To (\bigoplus_{\upsilon \in \Upsilon}F_\upsilon)_p$, where $\Lambda$ and $\Upsilon$ are sets, $\set{E_\lambda \mid \lambda \in \Lambda}$ and $\set{F_\upsilon \mid \upsilon \in \Upsilon}$ families of Banach spaces and $p=0$ or $1<p<\infty$. In this setting, for $\arr \subseteq \Lambda$ the operators $U_\arr$ and $P_\arr$ are defined as above (that is, $U_\arr$ and $P_\arr$ act to and from the domain of $T$, respectively). Similarly, for $\ess \subseteq \Upsilon$ we denote by $Q_\ess$ the canonical surjection of $(\bigoplus_{\upsilon \in \Upsilon}F_\upsilon)_p$ onto $(\bigoplus_{\upsilon \in \ess}F_\upsilon)_p$.

For a set $\Lambda$, a family of Banach spaces $\set{E_\lambda \mid \lambda \in \Lambda}$ and nonempty, bounded subsets $S_\lambda \subseteq E_\lambda $, $ \lambda \in \Lambda$, we say that $\set{S_\lambda \subseteq E_\lambda \mid \lambda \in \Lambda}$ is \emph{uniformly bounded} if $\sup\{ \abs{S_\lambda}\mid \lambda \in \Lambda\} <\infty$. If $\set{F_\lambda \mid \lambda \in \Lambda}$ is also a family of Banach spaces indexed by $\Lambda$, a set of operators $\set{T_\lambda \in \allop (E_\lambda , \, F_\lambda)\mid \lambda \in \Lambda}$ is said to be \emph{uniformly bounded} if $\sup\{\norm{T_\lambda}\mid \lambda \in \Lambda\}<\infty$. Given $1\leqslant p \leqslant \infty$ and a uniformly bounded family of operators $\set{T_\lambda \in \allop (E_\lambda , \, F_\lambda)\mid \lambda \in \Lambda}$, the $\ell_p$-\emph{direct sum of} $\set{T_\lambda \in \allop (E_\lambda , \, F_\lambda)\mid \lambda \in \Lambda}$, denoted $(\bigoplus_{\lambda \in \Lambda}T_\lambda)_p $, is the continuous linear map  that sends $(x_\lambda)_{\lambda \in \Lambda}\in (\bigoplus_{\lambda \in \Lambda}E_\lambda)_p $ to $(T_\lambda x_\lambda)_{\lambda \in \Lambda}\in (\bigoplus_{\lambda \in \Lambda}F_\lambda)_p$. Each of the operators $T_\lambda$ ($\lambda \in \Lambda$) is a \emph{summand} of the direct sum $(\bigoplus_{\lambda \in \Lambda}T_\lambda)_p $.

%Let $1\leqslant q<\infty$. We say that $p \in \set{0} \cup (1, \, \infty )$ is \emph{predual} to $q$ if it satisfies
%\[
%p = \left\{ \begin{tabular}{cl}$0$\, , & if $q=1$ \\ & \\ $\displaystyle \frac{q}{q-1}$\, , & if $1<q<\infty$ \end{tabular}  \right. .
%\]

A Banach space $E$ over the field $\real$ of real scalars is said to be \emph{Asplund} if every real-valued convex continuous function defined on a convex open subset $U$ of $E$ is {\fresh} differentiable on a dense $G_\delta$ subset of $U$. Our arguments hold for Banach spaces over the field $\field = \real$ or $\complex$; note that the notion of Asplund space may be extended (somewhat artificially) to complex Banach spaces by declaring a complex Banach space to be Asplund precisely when its underlying real Banach space structure is an Asplund space in the real scalar sense. By extending the notion of Asplund space to complex Banach spaces in this way, many of the well-known characterisations of Asplund spaces - for instance, a Banach space is Asplund if and only if each of its separable subspaces has separable dual \cite[Theorem~5.7]{Deville1993} - then hold also for complex Asplund spaces.

For Banach spaces $E$ and $F$, an operator $T:E\To F$ is \emph{Asplund} if for any finite positive measure space $(\Omega, \, \Sigma , \, \mu)$, any $S\in \allop (F, \,  L_\infty (\Omega, \, \Sigma , \, \mu))$ and any $\eps >0$, there exists $B\in \Sigma$ such that $\mu (B)>\mu (\Omega )-\eps$ and $\set{f\chi_B\mid f\in ST(\cball{E})}$ is relatively compact in $L_\infty (\Omega, \, \Sigma , \, \mu)$ (here $\chi_B$ denotes the characteristic function of $B$ on $\Omega$). We note that some authors, for example in \cite{Pietsch1980} and \cite{Heinrich1980}, refer to Asplund operators as \emph{decomposing operators}. Standard references for Asplund operators are \cite{Pietsch1980} and \cite{Stegall1981}, where it is shown that the Asplund operators form a closed operator ideal and that a Banach space is an Asplund space if and only if its identity operator is an Asplund operator. A further impressive result is that every Asplund operator factors through an Asplund space; this is due independently to O.~{\reinov} \cite{Reuinov1978}, S.~Heinrich \cite{Heinrich1980} and C.~Stegall \cite{Stegall1981}.

We now define the Szlenk index, noting that our definition varies from that given by W. Szlenk in \cite{Szlenk1968}. However, the two definitions give the same index for operators acting on separable Banach spaces containing no isomorphic copy of $\ell_1$ (see the proof of \cite[Proposition~3.3]{Lancien1993} for details). 

Let $E$ be a Banach space, $K \subseteq E\dual$ a $w\dual$-compact set and $\eps >0$. Define
\[
\mdset{\mbox{}}{\eps}{K} := \set{x \in K \mid \diam (K \cap V)> \eps \mbox{ for every } w\dual \mbox{-open }V\ni x}.
\] We iterate $s_\eps$ transfinitely as follows: let $\mdset{0}{\eps}{K} = K$, $ \mdset{\alpha+1}{\eps}{K}= \mdset{\mbox{}}{\eps}{\mdset{\alpha}{\eps}{K}}$ for each ordinal $\alpha$ and, if $\alpha$ is a limit ordinal, $ \mdset{\alpha}{\eps}{K} = \bigcap_{\beta< \alpha} \mdset{\beta}{\eps}{K}$. 

The \emph{$\eps$-Szlenk index of $K$}, denoted $\meeszlenk{\eps}{K}$, is the class of all ordinals $\alpha$ such that $\mdset{\alpha}{\eps}{K} \neq \emptyset$. The \emph{Szlenk index of $K$} is the class $\bigcup_{\eps >0}\meeszlenk{\eps}{K}$. Note that $\meeszlenk{\eps}{K}$ (resp., $\mszlenk{K}$) is either an ordinal or the class $\ord$ of all ordinals. If $\meeszlenk{\eps}{K}$ (resp., $\mszlenk{K}$) is an ordinal, then we write $\meeszlenk{\eps}{K} < \infty$ (resp., $\mszlenk{K}<\infty$), and otherwise we write $\meeszlenk{\eps}{K} = \infty$ (resp., $\mszlenk{K}=\infty$). For a Banach space $E$, the \emph{$\eps$-Szlenk index of $E$} is $\meeszlenk{\eps}{E} = \meeszlenk{\eps}{{\cball{E\dual}}}$, and the \emph{Szlenk index of $E$} is $\mszlenk{E} = \mszlenk{\cball{E\dual}}$. If $T: E\To F$ is an operator, the \emph{$\eps$-Szlenk index of $T$} is $\meeszlenk{\eps}{T} = \meeszlenk{\eps}{{T\dual\cball{F\dual}}}$, whilst the \emph{Szlenk index of $T$} is $\mszlenk{T} = \mszlenk{T\dual\cball{F\dual}}$.

It is clear that the Szlenk index of a nonempty $w\dual$-compact set cannot be $0$. We note also that, by $w\dual$-compactness, the $\eps$-Szlenk index of a nonempty $w\dual$-compact set $K$ is never a limit ordinal.

The following proposition states some known facts about the Szlenk index.

\begin{prop}\label{collection}
Let $E$ and $F$ be Banach spaces, $T: E\To F$ an operator and $K\subseteq E\dual$ a nonempty $w\dual$-compact set.

\nr{i} If $E$ is isomorphic to a quotient or subspace of $F$, then $ \mszlenk{E}\leqslant \mszlenk{F}$. In particular, the Szlenk index is an isomorphic invariant of a Banach space.

\nr{ii} $\mszlenk{E} < \infty$ if and only if $E$ is an Asplund space. Similarly, $\mszlenk{T} < \infty$ if and only if $T$ is an Asplund operator.

\nr{iii} If $K$ is absolutely convex and $\mszlenk{K}<\infty$, then there exists an ordinal $\alpha$ such that $\mszlenk{K} = \omega^\alpha$. In particular, the Szlenk index of an Asplund space or Asplund operator is of the form $\omega^\alpha$ for some (unique) ordinal $\alpha$.

\nr{iv} $\mszlenk{K} = 1$ if and only if $K$ is norm-compact. In particular, $\mszlenk{E}=1$ if and only if ${\rm dim}(E) < \infty$, and $\mszlenk{T} = 1$ if and only if $T$ is compact.

\nr{v} $\mszlenk{E\oplus F} = \max \set{\mszlenk{E}, \, \mszlenk{F}}$.
\end{prop}

Parts (i) of Proposition~\ref{collection} is discussed in \cite{H'ajek2007a}. Part (ii) is discussed in \cite{H'ajek2007a} in the case of spaces, and the more general case of operators is established in \cite[Proposition~2.10]{Brookerc}. Part (iii) was proved for $K=\cball{E\dual}$ in \cite{Lancien1996}; see also p.64 of \cite{H'ajek2008}. As the proof of the case $K=\cball{E\dual}$ relies only upon the fact that $\cball{E\dual}$ is convex and symmetric (that is, absolutely convex), the proof applies also to arbitrary absolutely convex $K$. Part (iv) is a consequence of the fact that a $w\dual$-compact set is norm-compact if and only if its relative $w\dual$ and norm topologies coincide (see, e.g., \cite[Corollary~3.1.14]{Engelking1989}), with the final assertion regarding operators requiring the use of Schauder's theorem. Part (v) is essentially Proposition~2.4 of \cite{H'ajek2007} (see also \cite[Proposition~14]{Odell2007} for the separable case), and will be improved upon in Theorem~\ref{accunibound} below.

\begin{defn}
For each ordinal $\alpha$, define $\szlenkop{\alpha} := \set{T\in \allop \mid \mszlenk{T} \leqslant \omega^\alpha} $.
\end{defn}

As noted in the introduction, elements of $\szlenkop{\alpha}$ are known as $\alpha$-\emph{Szlenk operators}. We have the following:
\begin{thm}[{\cite[Theorem~2.2]{Brookerc}}]\label{idealthm}
Let $\alpha$ be an ordinal. Then $\szlenkop{\alpha}$ is a closed operator ideal.
\end{thm}

\section{Main results}\label{mainstuffing}

It is obvious that a direct sum of operators factors any of its summands. Thus, since $\set{T\in \allop \mid \mszlenk{T} <\infty}$ is the operator ideal of Asplund operators (see Proposition~\ref{collection}(ii)), it is only interesting to consider the Szlenk index of a direct sum of operators in the case that all of the summands are Asplund. With this in mind, we henceforth consider direct sums of Asplund operators only.

\subsection{$\ell_1$-direct sums and $\ell_\infty$-direct sums}
The task of determining the Szlenk index of $\ell_1$-direct sums and $\ell_\infty$-direct sums of operators is made considerably easier by the fact that the Banach spaces $\ell_1$ and $\ell_\infty$ fail to be Asplund, for this ensures that the norms of the summand operators must exhibit $c_0$-like behaviour in order for the direct sum operator to be Asplund. More precisely, we have the following result. 
\begin{prop}\label{nonascase}
Let $\Lambda$ be a set, $\set{E_\lambda \mid \lambda \in \Lambda}$ and $\set{F_\lambda \mid \lambda \in \Lambda}$ families of Banach spaces, $\set{T_\lambda \in \allop (E_\lambda , \, F_\lambda)\mid \lambda \in \Lambda}$ a uniformly bounded family of Asplund operators and $p=1$ or $p=\infty$. The following are equivalent:
\begin{itemize}
\item[{\rm (i)}] $\mszlenk{(\bigoplus_{\lambda \in \Lambda}T_\lambda)_p} <\infty$ (that is, $(\bigoplus_{\lambda \in \Lambda}T_\lambda)_p$ is Asplund).
\item[{\rm (ii)}] $\mszlenk{(\bigoplus_{\lambda \in \Lambda}T_\lambda)_p} = \sup\{\mszlenk{T_\lambda}\mid \lambda \in \Lambda\}$.
\item[{\rm (iii)}] $(\norm{T_\lambda})_{\lambda \in \Lambda}\in c_0(\Lambda )$.
\end{itemize}
\end{prop}

\begin{proof}
We prove (iii) $\Rightarrow$ (ii) $\Rightarrow$ (i) $\Rightarrow$ (iii).

Suppose (iii) holds; we will show $\mszlenk{(\bigoplus_{\lambda \in \Lambda}T_\lambda)_p} = \sup\{\mszlenk{T_\lambda}\mid \lambda \in \Lambda\}$. By Proposition~\ref{collection}(iii) there exist ordinals $\alpha_\lambda$, $\lambda \in \Lambda$, with $\mszlenk{T_\lambda} = \omega^{\alpha_\lambda}$ for each $\lambda$. Let $\alpha_\Lambda = \sup\{\alpha_\lambda\mid\lambda \in \Lambda\}$, so that $\sup\{\mszlenk{T_\lambda}\mid \lambda \in \Lambda\} = \omega^{\alpha_\Lambda}$. To see that $(\bigoplus_{\lambda \in \Lambda}T_\lambda)_p\in \szlenkop{\alpha_\Lambda}$, for $n\in \nat$ and $\lambda \in \Lambda$ let 
\[
T_{\lambda , \, n} = \begin{cases}
T_{\lambda}& \text{if $\norm{T_\lambda}>1/n$},\\
0& \text{otherwise}
\end{cases}
\]
and $V_n = (\bigoplus_{\lambda \in \Lambda}T_{\lambda, \, n})_p$. Note that $\set{T_{\lambda , \, n}\mid \lambda \in \Lambda , \, n\in\nat} \subseteq \szlenkop{\alpha_\Lambda}$, hence $V_n\in \szlenkop{\alpha_\Lambda}$ also since each $V_n$ can be written as a (finite) sum of operators that factor some element of $\set{T_{\lambda , \, n}\mid \lambda \in \Lambda , \, n\in\nat}$. Moreover, we have $\norm{V_n -(\bigoplus_{\lambda \in \Lambda}T_\lambda)_p} \leqslant 1/n$ for each $n\in\nat$, hence $V_n \rightarrow (\bigoplus_{\lambda \in \Lambda}T_\lambda)_p$. As $V_n\in \szlenkop{\alpha_\Lambda}$ for all $n$ and $\szlenkop{\alpha_\Lambda}$ is closed (Theorem~\ref{idealthm}), $(\bigoplus_{\lambda \in \Lambda}T_\lambda)_p \in \szlenkop{\alpha_\Lambda}$. In particular, $\mszlenk{(\bigoplus_{\lambda \in \Lambda}T_\lambda)_p}\leqslant \omega^{\alpha_\Lambda}=\sup\{\mszlenk{T_\lambda}\mid\lambda \in \Lambda\}$. The reverse inequality follows by Theorem~\ref{idealthm} and the fact that $(\bigoplus_{\lambda \in \Lambda}T_\lambda)_p$ factors each of the operators $T_\lambda$, $\lambda \in \Lambda$. We have now shown (iii) $\Rightarrow$ (ii).

It is trivial that (ii) $\Rightarrow$ (i), so remains only to show that (i) $\Rightarrow$ (iii). To this end, suppose that (iii) does not hold. Then there exists $\delta > 0$ and an infinite set $\Lambda ' \subseteq \Lambda$ such that $\norm{T_\lambda} >\delta$ for all $\lambda \in \Lambda '$, and so $(\bigoplus_{\lambda \in \Lambda}T_\lambda)_p$ factors an isomorphic embedding of the non-Asplund space $\ell_p$. By Proposition~\ref{collection}(ii), $\mszlenk{(\bigoplus_{\lambda \in \Lambda}T_\lambda)_p} = \infty$.
\end{proof}

\subsection{$c_0$-direct sums and $\ell_p$-direct sums $(1<p<\infty)$}\label{canemace}

In this section we consider the Szlenk index of a direct sum operator $(\bigoplus_{\lambda \in \Lambda}T_\lambda)_p$, where $p=0$ or $1<p<\infty$. As in the cases $p=1$ and $p=\infty$, if $(\norm{T_\lambda})_{\lambda \in\Lambda} \in c_0 (\Lambda)$ then $\mszlenk{(\bigoplus_{\lambda \in \Lambda}T_\lambda)_p} = \sup\{\mszlenk{T_\lambda}\mid \lambda\in\Lambda\}$. However, the situation is not so clear if $(\norm{T_\lambda})_{\lambda \in\Lambda} \notin c_0 (\Lambda)$, and we demonstrate this by way of an example. For an ordinal $\gamma$, we may equip the ordinal $\gamma +1$ with its order topology, thereby making it a compact Hausdorff space. C.~Samuel has shown that for each $\alpha <\omega_1$, $\mszlenk{C(\omega^{\omega^\alpha}+1)} = \omega^{\alpha+1}$ (Samuel's calculation is found in \cite{Samuel1984}, however a more direct approach has been discovered by P.~{\hajek} and G.~Lancien \cite{H'ajek2007}). By the Bessaga-{\pelch} linear isomorphic classification of $C(K)$ spaces with $K$ countable \cite[Theorem~1]{Bessaga1960}, $C(\omega^n +1)$ is linearly isomorphic to $C(\omega+1)$ for all $0<n<\omega$. Thus, in particular, $\mszlenk{C(\omega^n +1)} = \mszlenk{C(\omega+1)} = \omega$ for all $0<n<\omega$. For each $0<n<\omega$, let $T_n$ denote the identity operator on $C(\omega^n +1)$. As $(\bigoplus_{0<n<\omega}C(\omega^n+1))_0$ is linearly isomorphic to $C(\omega^\omega+1)$, by Samuel's result we have
\[
\textstyle \mszlenk{(\bigoplus_{0<n<\omega}T_n)_0} = \mszlenk{C(\omega^\omega+1)} = \omega^2 >\omega = \sup\{\mszlenk{T_n}\mid 0<n<\omega\}\, .
\]
Thus the situation under consideration in this section is more subtle than the cases of $\ell_1$-direct sums and $\ell_\infty$-direct sums. Our goal is to determine precisely the Szlenk index of a $c_0$-direct sum or $\ell_p$-direct sum ($1<p<\infty$) of operators in terms of the overall behaviour of the $\eps$-Szlenk indices of the summand operators. To this end, we now introduce some notation.

Given a set $\Lambda$, a family of Banach spaces $\set{E_\lambda \mid \lambda \in \Lambda}$, a corresponding uniformly bounded family $\set{K_\lambda \subseteq E_\lambda\dual \mid \lambda \in \Lambda}$ of absolutely convex, $w\dual$-compact sets and $1\leqslant q <\infty$, we define
\[
B_q(K_\lambda \mid \lambda \in \Lambda):= \bigcup_{(a_\lambda)_{\lambda \in \Lambda} \in \cball{\ell_q(\Lambda)}} \prod_{\lambda\in\Lambda}a_\lambda K_\lambda \, ,
\]
and always consider $B_q(K_\lambda \mid \lambda \in \Lambda)$ as a subset of $(\bigoplus_{\lambda \in \Lambda}E_\lambda)_p\dual$, where $p$ is predual to $q$ (recall from Section~\ref{prelimlim} that $(\bigoplus_{\lambda \in\Lambda}E_\lambda)\dual_p$ is naturally identified with $(\bigoplus_{\lambda \in\Lambda}E_\lambda\dual)_q$). Such a set $B_q(K_\lambda \mid \lambda \in \Lambda)$ so defined is clearly bounded, and it is not difficult to see that it is also $w\dual$-compact. Indeed, for each $\lambda\in\Lambda$ define $T_\lambda : E_\lambda \longrightarrow C(K_\lambda )$ to be the map that sends $x\in E_\lambda$ to the continuous function $k \mapsto \functional{k}{x}$ ($k\in K_\lambda$). Then the {\kreinmilman} theorem, along with other classical results regarding extreme points (see, for example, \cite[Lemma~3.42]{Fabian2001} and \cite[Exercise~2.4]{Giles1982}), implies that $T_\lambda \dual \cball{C(K_\lambda )\dual} = K_\lambda$ for each $\lambda \in \Lambda$. Hence $B_q(K_\lambda \mid \lambda \in \Lambda) = (\bigoplus_{\lambda \in \Lambda}T_\lambda)_{p}\dual \cball{(\bigoplus_{\lambda \in \Lambda}C(K_\lambda ))_{p}\dual}$, ensuring the $w\dual$-compactness of $B_q(K_\lambda \mid \lambda \in \Lambda)$.

We first deal explicitly with the case where the Szlenk index of a direct sum of operators  has Szlenk index $\omega^0 = 1$. The following result describes the situation for this case.

\begin{prop}\label{compactbound}
Let $\Lambda$ be a set, $\set{E_\lambda \mid \lambda \in \Lambda}$ and $\set{F_\lambda \mid \lambda \in \Lambda}$ families of Banach spaces, $\set{T_\lambda \in \allop (E_\lambda , \, F_\lambda) \mid \lambda \in \Lambda}$ a uniformly bounded family of operators and $p\in \set{0} \cup [1, \, \infty]$. The following are equivalent:
\begin{itemize}
\item[{\rm (i)}] $\mszlenk{(\bigoplus_{\lambda \in \Lambda}T_\lambda)_p}=1$.
\item[{\rm (ii)}] $\mszlenk{T_\lambda} = 1$ for every $\lambda \in \Lambda$ and $(\norm{T_\lambda})_{\lambda \in \Lambda} \in c_0 (\Lambda)$.
\end{itemize}
\end{prop}

Proposition~\ref{compactbound} follows immediately from Proposition~\ref{collection}(iv) and the following proposition.

\begin{prop}\label{compacteasy}
Let $\Lambda$ be a set, $\set{E_\lambda \mid \lambda \in \Lambda}$ and $\set{F_\lambda \mid \lambda \in \Lambda}$ families of Banach spaces, $\set{T_\lambda \in \allop (E_\lambda , \, F_\lambda) \mid \lambda \in \Lambda}$ a uniformly bounded family of operators and $p\in \set{0} \cup [1, \, \infty]$. The following are equivalent:
\begin{itemize}
\item[{\rm (i)}] $(\bigoplus_{\lambda \in \Lambda}T_\lambda)_p$ is compact.
\item[{\rm (ii)}] $T_\lambda$ is compact for every $\lambda \in \Lambda$ and $(\norm{T_\lambda})_{\lambda \in \Lambda} \in c_0 (\Lambda)$.
\end{itemize}
\end{prop}

We omit the straightforward proof of Proposition~\ref{compacteasy}, but note that it is similar to the proof of Proposition~\ref{nonascase} presented earlier.

The general case for $c_0$-direct sums and $\ell_p$-direct sums of operators, where $1<p<\infty$, will be deduced from the following key result.

\begin{prop}\label{unibound}
Let $\Lambda$ be a set, $\set{E_\lambda \mid \lambda \in \Lambda}$ a family of Banach spaces, $\set{ K_\lambda \subseteq E_\lambda\dual \mid \lambda \in \Lambda, \, K_\lambda \neq \emptyset}$ a uniformly bounded family of nonempty absolutely convex $w\dual$-compact sets, $\alpha >0$ an ordinal and $1\leqslant q < \infty$. The following are equivalent:
\begin{itemize}
\item[{\rm (i)}] $\mszlenk{B_q(K_\lambda \mid \lambda \in \Lambda)} \leqslant \omega^{\alpha}$.
\item[{\rm (ii)}] $\sup\{\meeszlenk{\eps}{K_\lambda}\mid \lambda \in \Lambda\}< \omega^\alpha$ for every $\eps >0$.
\item[{\rm (iii)}] $\sup\{\meeszlenk{\eps}{B_q(K_\lambda \mid \lambda \in \eff)}\mid \eff \in \Lambda\fin\}< \omega^\alpha$ for every $\eps > 0$.
\end{itemize}
\end{prop}

To establish Proposition~\ref{unibound}, we prove (i) $\Rightarrow$ (ii) $\Rightarrow$ (iii) $\Rightarrow$ (i). In proving the implication (ii) $\Rightarrow$ (iii), we shall call upon the following technical result:

\begin{lemma}\label{frount}
Let $E_1, \ldots , E_n$ be Banach spaces, $K_1 \subseteq E_1\dual , \ldots , K_n\subseteq E_n\dual$ nonempty, absolutely convex, $w\dual$-compact sets, $1\leqslant q < \infty$ and $\eps >0$. Let $\displaystyle d=\max\{\diam (K_i) \mid 1\leqslant i \leqslant n\}$ and let $m$ and $ M$ be natural numbers such that $M\geqslant m\geqslant 2$ and $(2^q-1)\eps^qM\geqslant 8^q d^q (m-1)$. Suppose $\alpha$ is an ordinal such that $s^{\omega^\alpha \cdot M}_{\eps}(B_q(K_i\mid 1\leqslant i \leqslant n)) \neq \emptyset$. Then, for every $\delta \in (0, \, \eps /16)$ there is $i \leqslant n$ such that $\mdset{\omega^\alpha \cdot m}{\delta}{K_i} \neq \emptyset$.
\end{lemma}
%Let $E_1, \ldots , E_n$ be Banach spaces, $K_1 \subseteq E_1\dual , \ldots , K_n\subseteq E_n\dual$ nonempty, absolutely convex, $w\dual$-compact sets, $\alpha$ an ordinal, $1\leqslant q < \infty$ and $\eps >0$. Let $\displaystyle d=\max_{1\leqslant i \leqslant n}\diam (K_i) $ and let $m, \, M\in \nat$ be such that $M\geqslant m\geqslant 2$ and $(2^q-1)\eps^qM\geqslant 8^q d^q (m-1)$. Suppose that $s^{\omega^\alpha \cdot M}_{\eps}(B_q(K_i\mid 1\leqslant i \leqslant n)) \neq \emptyset$. Then for every $\delta \in (0, \, \eps /16)$ there is $i \leqslant n$ such that $\mdset{\omega^\alpha \cdot m}{\delta}{K_i} \neq \emptyset$.

The proof of Lemma~\ref{frount} is delayed until Section~\ref{techsect}. To show (iii) $\Rightarrow $ (i) we require the following discrete variant of \cite[Lemma~3.3]{H'ajek2007}:

\begin{lemma}\label{tvl}
Let $\Lambda$ be a set, $(E_\lambda)_{\lambda \in \Lambda}$ a family of Banach spaces, $1\leqslant q<\infty$, $p$ predual to $q$ and $K \subseteq (\bigoplus_{\lambda \in \Lambda}E_\lambda)_{p}\dual $ nonempty and $w\dual$-compact. Let $\alpha$ be an ordinal, $\arr \subseteq \Lambda$ and $\eps > \delta >0$. If $x \in \mdset{\alpha}{\eps}{K}$ and $\norm{U_\arr\dual \, x}^{q} > \abs{K}^q-(\frac{\eps - \delta}{2})^q$, then $U_\arr\dual\, x \in \mdset{\alpha}{\delta}{U_\arr\dual\, K}$.
\end{lemma}

\begin{proof}
We fix $\eps$, $\delta$ and $\arr$ and proceed by induction on $\alpha$. The conclusion of the lemma is trivially true for $\alpha = 0$. So suppose that $\beta$ is an ordinal such that the conclusion of the lemma holds with $\alpha = \beta$; we show that it holds then also for $\alpha = \beta +1$. To this end, let $x \in K$ be such that $\norm{U_\arr\dual \, x}^{q} > \abs{K}^q-(\frac{\eps - \delta}{2})^q$ and $U_\arr\dual\, x \notin \mdset{\beta+1}{\delta}{U_\arr\dual \, K}$. Our goal is to show that $x \notin \mdset{\beta +1}{\eps}{K}$, so we may assume that $x \in \mdset{\beta}{\eps}{K}$, hence $U_\arr\dual\, x \in \mdset{\beta}{\delta}{U_\arr\dual \, K}$ by the inductive hypothesis. It follows that there is $w\dual$-open $V \subseteq (\bigoplus_{\lambda \in \arr}E_\lambda)_{p}\dual $ such that $U_\arr\dual\, x \in V$ and $d:= \diam (V\cap \mdset{\beta}{\delta }{U_\arr\dual\, K}) \leqslant \delta $.
As $U_\arr\dual\, x$ does not belong to the $w\dual$-closed set $(\abs{K}^q-(\frac{\eps - \delta}{2})^q)^{1/q}\cball{(\bigoplus_{\lambda \in \arr}E_\lambda)_{p}\dual}$, we may assume \[ \textstyle V \cap \big(\abs{K}^q-(\frac{\eps - \delta}{2})^q\big)^{1/q}\cball{(\bigoplus_{\lambda \in \arr}E_\lambda)_{p}\dual} = \emptyset .\] Let $W = (U_\arr\dual)^{-1} (V)$ and let $u \in W \cap \mdset{\beta}{\eps}{K}$. Then $\norm{U_\arr\dual \, u}^{q} > \abs{K}^q-(\frac{\eps - \delta}{2})^q$ and $u \in \mdset{\beta}{\eps}{K}$, hence by the induction hypothesis $U_\arr\dual \, u \in V \cap \mdset{\beta}{\delta}{U_\arr\dual \, K}$. So for $u_1 , \, u_2 \in W \cap \mdset{\beta}{\eps}{K}$ we have $\norm{U_\arr \dual \, u_1 - U_\arr \dual \, u_2}^{q} \leqslant d^{q} \leqslant \delta^q$. Moreover, since $\norm{U_\arr \dual \, u_1 }^q > \abs{K}^q-(\frac{\eps - \delta}{2})^q$ it follows that
\[
\norm{u_1 - P_\arr\dual U_\arr\dual u_1} \leqslant \left( \abs{K}^q - \norm{P_\arr\dual U_\arr\dual u_1}^q \right)^{1/q} = \left( \abs{K}^q - \norm{U_\arr\dual u_1}^q \right)^{1/q} < \frac{\eps - \delta}{2}\, .
\]
Similarly, $\norm{u_2 - P_\arr\dual U_\arr\dual u_2} < \frac{\eps - \delta}{2}$. We now deduce that
\begin{align*}
\Vert u_1 -& u_2\Vert^{q}\\ &= \norm{P_\arr\dual U_\arr \dual \, u_1 - P_\arr\dual U_\arr \dual \, u_2}^q + \norm{(u_1 - P_\arr\dual U_\arr \dual \, u_1) -(u_2 - P_\arr\dual U_\arr \dual \, u_2)}^q \\ & \leqslant \norm{U_\arr \dual \, u_1 - U_\arr \dual \, u_2}^q + \left( 2\cdot \frac{\eps - \delta}{2}\right)^q \\ &\leqslant \delta^q + (\eps-\delta)^q \\ &\leqslant \eps^q \, .
\end{align*}
In particular, $\diam (W\cap \mdset{\beta}{\eps}{K})\leqslant \eps$. It follows that $x \notin  \mdset{\beta +1}{\eps}{K}$, as desired.

The lemma passes easily to limit ordinals, so we are done.
\end{proof}

In order to state the third (and final) lemma required in the proof of Proposition~\ref{unibound}, we give the following definition.

\begin{defn}\label{postdoc1}
For real numbers $a\geqslant 0$, $b>c>0$ and $1\leqslant d<\infty$, define
\[
\sigma (a,b,c,d):= \inf \bigg\{ n\in \nat \,\, \bigg\vert \,\,  n \geqslant \bigg(\frac{2a}{b-c}\bigg)^{\! d} - \bigg(\frac{b}{b-c}\bigg)^{\! d}+1\bigg\} \, .
\]
\end{defn}

With regards to Definition~\ref{postdoc1}, note that $\sigma (a,b,c,d)=1$ whenever $2a\leqslant b$.

\begin{lemma}\label{postdoc2}
Let $\Lambda$ be a set, $\{ E_\lambda \mid \lambda\in\Lambda \}$ a family of Banach spaces, $1\leqslant q <\infty$, $p$ predual to $q$, $K\subseteq (\bigoplus_{\lambda\in\Lambda}E_\lambda )_p\dual$ a nonempty, $w\dual$-compact set and $\eps>\delta >0$. Suppose $\eta_\delta $ is a nonzero ordinal such that $s_\delta^{\eta_\delta}(U_\eff\dual K) = \emptyset$ for every $\eff\in\Lambda\fin$. Then $s_\eps^{\eta_\delta \cdot \sigma (\vert K\vert , \eps , \delta , q)} = \emptyset$, hence ${\rm Sz}_\eps (K)\leqslant \eta_\delta \cdot \sigma (\vert K\vert , \eps , \delta , q)$.
\end{lemma}

\begin{proof} We \emph{claim} that for each $n<\omega$, either $s_{\eps}^{\eta_\delta \cdot n}(K)$ is empty or
\begin{equation}\label{fewmorethings}
\abs{s_{\eps}^{\eta_\delta \cdot n}(K)}^q \leqslant  \abs{K}^q -n\left( \frac{\eps-\delta}{2} \right)^q \, .
\end{equation}
To prove the claim, we proceed by induction on $n$. (\ref{fewmorethings}) holds trivially for $n=0$. Suppose the claim holds for $n=m$; we will show that it holds for $n=m+1$. For every $\eff\in\Lambda\fin$ we have
\begin{equation}\label{evenmore}
s_\delta^{\eta_\delta}( U_\eff\dual s_\eps^{\eta_\delta \cdot (m+1)}(K)) \subseteq s_\delta^{\eta_\delta}(U_\eff\dual K) = \emptyset\, .
\end{equation}
If $s_\eps^{\eta_\delta \cdot m} (K)= \emptyset$, we are done. Otherwise, by the induction hypothesis,
\begin{equation}\label{yetmore}
\abs{s_\eps^{\eta_\delta \cdot (m+1)} (K)}^q \leqslant  \abs{K}^q -m\left( \frac{\eps-\delta}{2}\right)^q \, .
\end{equation}
If $s_\eps^{\eta_\delta \cdot (m+1)} (K) \neq \emptyset$, then applying (\ref{evenmore}), (\ref{yetmore}) and Lemma~\ref{tvl} implies that for every $x\in s_\eps^{\eta_\delta \cdot (m+1)} (K)$ and $\eff \in \Lambda\fin$, we have
\[
\norm{U_\eff \dual \, x}^q \leqslant \abs{K}^q -m\left( \frac{\eps-\delta}{2} \right)^q - \left( \frac{\eps-\delta}{2}\right)^q = \abs{K}^q -(m+1)\left( \frac{\eps-\delta}{2}\right)^q \, .
\]
Thus $x\in s_\eps^{\eta_\delta \cdot (m+1)} (K)$ implies
\[
\norm{x}^q = \sup\{\norm{U_\eff\dual \, x}^q \mid \eff \in \Lambda\fin\}\leqslant \abs{K}^q -(m+1)\left(\frac{\eps-\delta}{2}\right)^q\, ,
\]
and so (\ref{fewmorethings}) holds for $n=m+1$. The inductive proof of the claim is complete.

By definition (precisely, Definition~\ref{postdoc1}), we have
\begin{equation}\label{postdoc3}
\abs{K}^q - (\sigma (\abs{K}, \eps, \delta , q)-1)\left(\frac{\eps-\delta}{2}\right)^q \leqslant \left(\frac{\eps}{2}\right)^q\, .
\end{equation}
Thus, by (\ref{postdoc3}) and the claim proved above we have
\[
\diam (s_\eps^{\eta_\delta \cdot (\sigma (\abs{K},\eps,\delta,q)-1)}(K)) \leqslant 2\cdot \frac{\eps}{2} = \eps\, ,
\]
and we thus deduce that
\[
s_\eps^{\eta_\delta \cdot \sigma (\abs{K},\eps,\delta,q)}(K) \subseteq s_\eps^{\eta_\delta \cdot (\sigma (\abs{K},\eps,\delta,q)-1)+1}(K) = s_\eps (s_\eps^{\eta_\delta \cdot (\sigma (\abs{K},\eps,\delta,q)-1)}(K)) = \emptyset \, .\qedhere
\]
\end{proof}

We now give the proof of Proposition~\ref{unibound}, assuming Lemma~\ref{frount}.

\begin{altproof3} We prove (i) $\Rightarrow$ (ii) $\Rightarrow$ (iii) $\Rightarrow$ (i). Throughout, $p$ shall denote the real number predual to $q$.

To show (i) $\Rightarrow$ (ii), suppose by way of a contraposition that there is $\eps >0$ such that $\sup\{ \meeszlenk{\eps}{K_\lambda}\mid \lambda \in \Lambda\}\geqslant \omega^\alpha$. For each $\lambda' \in \Lambda$, the restriction $P_{\set{\lambda'}}\dual |_{{K_{\lambda'}}}$ is a norm-isometric, $w\dual$-homeomorphic embedding of $K_{\lambda'}$ into $B_q(K_\lambda \mid \lambda \in \Lambda)$, hence $\meeszlenk{\delta}{B_q(K_\lambda \mid \lambda \in \Lambda)} \geqslant \meeszlenk{\delta}{K_{\lambda'}}$ for all $\delta >0$ and $\lambda' \in \Lambda$. Thus
\begin{equation}\label{famcommday}
\meeszlenk{\eps}{B_q(K_\lambda \mid \lambda \in \Lambda)} \geqslant \sup\{ \meeszlenk{\eps}{K_\lambda} \mid \lambda \in \Lambda\} \geqslant \omega^\alpha .
\end{equation}
As $\meeszlenk{\eps}{B_q(K_\lambda \mid \lambda \in \Lambda)}$ cannot be a limit ordinal, we deduce from (\ref{famcommday}) that \[ \mszlenk{B_q(K_\lambda \mid \lambda \in \Lambda)} \geqslant \meeszlenk{\eps}{B_q(K_\lambda \mid \lambda \in \Lambda)} > \omega^\alpha .\] This proves (i) $\Rightarrow$ (ii).

Suppose (ii) holds. For each $\eps >0$ let $1<m_\eps <\omega$ and $\beta_\eps  < \alpha$ be such that $ \sup\{ \meeszlenk{\eps/32}{K_\lambda}\mid \lambda \in \Lambda\}< \omega^{\beta_{\eps} \cdot m_\eps }$. Set $d= \sup\{ \diam (K_\lambda)\mid \lambda \in \Lambda\}$ and for each $\eps \in (0, \, 1)$ let $M_\eps \in \nat$ be such that $(2^q-1)\eps^q {M_\eps} \geqslant 8^qd^q(m_\eps -1)$. By Lemma~\ref{frount}, for $\eff \in \Lambda\fin$ we have $\meeszlenk{\eps}{B_q(K_\lambda \mid \lambda \in \eff)} <\omega^{\beta_{\eps} }\cdot {M_\eps}$, hence
\[
\sup\{\meeszlenk{\eps}{B_q(K_\lambda \mid \lambda \in \Lambda)}\mid \eff \in \Lambda\fin\} \leqslant \omega^{\beta_{\eps} }\cdot {M_\eps} < \omega^\alpha .
\]
Thus, (ii) $\Rightarrow$ (iii).

Suppose that (iii) holds. As $U_\eff\dual B_q (K_\lambda \mid \lambda \in \Lambda) = B_q(K_\lambda \mid \lambda\in\eff)$ for each $\eff\in\Lambda\fin$, applying Lemma~\ref{postdoc2} with $K= B_q (K_\lambda \mid \lambda \in \Lambda) $, $\delta =\delta(\eps) = \eps/2$ and $\eta_{\delta (\eps)}= \sup \{ {\rm Sz}_{\eps/2}(U_\eff\dual B_q (K_\lambda \mid \lambda \in \Lambda) )\mid \eff\in\Lambda\fin \} \, \, (<\omega^\alpha)$ yields
\begin{align*}
{\rm Sz}(B_q (K_\lambda \mid \lambda \in \Lambda) ) &= \sup \{{\rm Sz}_\eps (B_q (K_\lambda \mid \lambda \in \Lambda)) \mid \eps >0 \}\\ &\leqslant \sup \big\{\eta_{\delta (\eps)}\cdot \sigma (\sup \{ \vert K_\lambda \vert \mid \lambda \in \Lambda \},\, \eps, \, \eps/2,\, q) \, \big\vert \, \eps >0 \big\}\\ &\leqslant \omega^\alpha \, ,
\end{align*}
hence (iii)$\Rightarrow$(i).
\end{altproof3}

\begin{remark}
The idea that an iterated implementation of Lemma~\ref{tvl} (c.f. Lemma~\ref{postdoc2} and its proof) might be used to prove the implication (iii)$\Rightarrow$(i) in the proof of Proposition~\ref{unibound} was essentially suggested to the author by Professor Gilles Lancien; previous versions of the main results of this chapter used a slightly different argument (also using Lemma~\ref{tvl}, but just a single direct application) and required the additional hypothesis that $K_\lambda = \cball{E_\lambda\dual}$ for all $\lambda$ (see Theorem~\ref{accunibound}).\end{remark}

The following result, along with Proposition~\ref{compactbound}, determines precisely the Szlenk index of a $c_0$-direct sum or $\ell_p$-direct sum of operators ($1<p<\infty$) in terms of properties of the $\eps$-Szlenk indices of the summands.

\begin{thm}\label{punibound}
Let $\Lambda$ be a set, $\set{E_\lambda \mid \lambda \in \Lambda}$ and $\set{F_\lambda \mid \lambda \in \Lambda}$ families of Banach spaces, $\set{T_\lambda :E_\lambda \To F_\lambda \mid \lambda \in \Lambda}$ a uniformly bounded family of Asplund operators, $\alpha >0$ an ordinal and $p=0$ or $ 1<p< \infty$. The following are equivalent:
\begin{itemize}
\item[{\rm (i)}] ${\rm Sz}\big( (\bigoplus_{\lambda \in \Lambda}T_\lambda)_p\big) \leqslant \omega^{\alpha}$.
\item[{\rm (ii)}] $\sup\{ \meeszlenk{\eps}{T_\lambda}\mid \lambda \in \Lambda\}< \omega^\alpha$ for all $\eps >0$.
\end{itemize}
It follows that if $T$ is noncompact, then \[ \textstyle {\rm Sz}\big( (\bigoplus_{\lambda \in \Lambda}T_\lambda)_p\big) = \inf \set{\omega^\alpha \mid \sup\{ \meeszlenk{\eps}{T_\lambda}\mid \lambda \in \Lambda\}< \omega^\alpha \mbox{ for all }\eps >0} .\]
\end{thm}

\begin{proof}
For convenience we set $T = (\bigoplus_{\lambda \in \Lambda}T_\lambda)_p$. The equivalence of (i) and (ii) is achieved by applying Proposition~\ref{unibound} with $K_\lambda = T_\lambda\dual\cball{F_\lambda\dual}$ for all $\lambda \in \Lambda$, for in this case $T\dual \cball{(\bigoplus_{\lambda \in \Lambda}E_\lambda)_p\dual} =B_{q}(T_\lambda\dual\cball{F_\lambda\dual} \mid \lambda \in \Lambda)$, where $q\in [1, \, \infty )$ is dual to $p$.

For each $\lambda \in \Lambda$ let $\alpha_\lambda$ denote the unique ordinal satisfying $\mszlenk{T_\lambda} = \omega^{\alpha_\lambda}$. Set $\alpha_\Lambda = \sup\{\alpha_\lambda\mid \lambda \in \Lambda\}$ and note that the set \[ \set{\omega^\alpha \mid \sup\{\meeszlenk{\eps}{T_\lambda}\mid \lambda \in \Lambda\}< \omega^\alpha \mbox{ for all }\eps >0} \ni \omega^{\alpha_\Lambda +1}\] is nonempty. We have $ \mszlenk{T} \leqslant \inf \set{\omega^\alpha \mid \sup_{\lambda \in \Lambda}\meeszlenk{\eps}{T_\lambda}< \omega^\alpha \mbox{ for all }\eps >0}$ by the implication (ii) $\Rightarrow$ (i) above. 

To complete the proof, we now suppose that $T$ is noncompact. As $\mszlenk{T}$ is a power of $\omega$, it is enough to show that $\mszlenk{T}>\omega^\beta$ holds for $\beta$ satisfying $\omega^\beta < \inf \set{\omega^\alpha \mid \sup\{\meeszlenk{\eps}{T_\lambda}\mid \lambda \in \Lambda\}< \omega^\alpha \mbox{ for all }\eps >0}$. Take such $\beta$. If $\beta = 0$, then $\mszlenk{T}>\omega^\beta$ by noncompactness of $T$. On the other hand, if $\beta>0$ then there is $\eps>0$ so small that $\meeszlenk{\eps}{T}\geqslant \sup\{\meeszlenk{\eps}{T_\lambda}\mid \lambda \in \Lambda\} \geqslant \omega^\beta$. As $\meeszlenk{\eps}{T}$ cannot be a limit ordinal, we conclude that $\mszlenk{T}\geqslant \meeszlenk{\eps}{T}>\omega^\beta$.
\end{proof}

\subsection{Applications}\label{airconfix}
Our first result here is the following Banach space analogue of Theorem~\ref{punibound} which determines precisely the Szlenk index of a $c_0$-direct sum or $\ell_p$-direct sum of Banach spaces in terms of the behaviour of the $\eps$-Szlenk indices of the summand spaces.
\begin{thm}\label{accunibound}
Let $\Lambda$ be a set, $\set{E_\lambda \mid \lambda \in \Lambda}$ a family of Asplund spaces, $\alpha >0$ an ordinal and $p=0$ or $1<p<\infty$. The following are equivalent:
\begin{itemize}
\item[{\rm (i)}] ${\rm Sz}\big( (\bigoplus_{\lambda \in \Lambda}E_\lambda)_p \big) \leqslant \omega^{\alpha}$.
\item[{\rm (ii)}] $\sup \{ \meeszlenk{\eps}{E_\lambda}\mid \lambda \in \Lambda\}< \omega^\alpha$ for all $\eps >0$.
\end{itemize}
It follows that if $(\bigoplus_{\lambda \in \Lambda}E_\lambda)_p$ is infinite dimensional, then \[
\textstyle {\rm Sz}\big( (\bigoplus_{\lambda \in \Lambda}E_\lambda)_p \big) = \inf \set{\omega^\alpha \mid \sup\{\meeszlenk{\eps}{E_\lambda}\mid \lambda \in \Lambda\}< \omega^\alpha \mbox{ for all }\eps>0}\, .
\]
\end{thm}

\begin{proof}
The conclusions of the theorem follow by taking $T_\lambda$ to be the identity operator of $E_\lambda$ for each $\lambda \in \Lambda$ in the statement of Theorem~\ref{punibound}.
\end{proof}

\begin{thm}\label{trivcase}
Let $\Lambda$ be a set, $E$ an infinite dimensional Banach space and $1< p, r < \infty$. Then \[ \mszlenk{E} = \mszlenk{c_0 (\Lambda , \, E)} = \mszlenk{\ell_p (\Lambda , \, E)} = \mszlenk{\ell_r (\Lambda , \, E)} .\]
\end{thm}

\begin{proof}
Apply Theorem~\ref{accunibound} with $E_\lambda = E$ for all $\lambda \in \Lambda$.
\end{proof}

The previous theorem, Theorem~\ref{trivcase}, allows us to add to the class of ordinals $\gamma$ for which the Szlenk index of $C(\gamma +1)$ is known (here, $\gamma +1$ is equipped with its order topology). The computation of the Szlenk index of $C(\omega_1+1)$, in particular $\mszlenk{C(\omega_1+1)}=\omega_1\cdot \omega$, is due to {\hajek} and Lancien \cite{H'ajek2007}. Essentially using the fact that $\mszlenk{C(\xi +1)} = \mszlenk{C(\zeta +1)}$ for ordinals $\xi$ and $\zeta$ satisfying $\xi \leqslant \zeta <\xi\cdot\omega$ (an easy consequence of Proposition~\ref{collection}(v)), {\hajek} and Lancien deduce that $\mszlenk{C(\gamma+1)}=\omega_1\cdot \omega$ whenever $\omega_1\leqslant \gamma <\omega_1\cdot\omega$. We claim that $\mszlenk{C(\gamma+1)}=\omega_1\cdot \omega$ whenever $\omega_1\leqslant \gamma <\omega_1\cdot\omega^\omega$, a fact that will follow once we have shown that $\mszlenk{C(\xi +1)} = \mszlenk{C(\zeta +1)}$ whenever $\xi$ and $\zeta$ are ordinals satisfying $\omega \leqslant \xi \leqslant \zeta <\xi\cdot\omega^\omega$. If $\xi$ and $\zeta$ are ordinals satisfying $\omega \leqslant \xi \leqslant \zeta <\xi\cdot\omega^\omega$, then there exists $n<\omega$ such that $C(\zeta +1)$ is isomorphic to a subspace of $C(\xi \cdot \omega^n +1)$. Thus, by Proposition~\ref{collection}(i), it suffices to show that $\mszlenk{C(\xi+1)} = \mszlenk{C(\xi \cdot \omega^n +1)}$ for all $n<\omega$. This is obviously true for $n=0$, and if true for some $n$ then, since $C(\xi\cdot\omega^{n+1}+1)$ is isomorphic to $c_0(\omega,\, C(\xi\cdot\omega^n+1))$, Theorem~\ref{trivcase} yields
\begin{align*}
\mszlenk{C(\xi\cdot\omega^{n+1}+1)} %= \mszlenk{C_0(\xi\cdot\omega^{n+1})} 
= \mszlenk{c_0(\omega,\, C(\xi\cdot\omega^n+1))} &= \mszlenk{C(\xi\cdot\omega^n+1)} \\&= \mszlenk{C(\xi+1)},
\end{align*}
which completes the proof.

The following proposition asserts that the set of all countable values of the Szlenk index of Banach spaces is attained by the class of Banach spaces with a shrinking basis. A further consequence of this result is that if for $\alpha <\omega_1$ there exists a Banach space of Szlenk index $\omega^\alpha$, then {\pelch}'s complementably universal basis space (see \cite{Pelczy'nski1969}) has a complemented subspace of Szlenk index $\omega^\alpha$.

\begin{prop}\label{shrinkex}
Let $0<\alpha <\omega_1$. The following are equivalent:

\nr{i} There exists a Banach space $E$ with $\mszlenk{E} =\omega^\alpha$.

\nr{ii} There exists a Banach space $E$ with a shrinking basis and $\mszlenk{E} =\omega^\alpha$.
\end{prop}

To prove Proposition~\ref{shrinkex}, we shall call on the following result regarding subspaces and quotients, due to G. Lancien \cite{Lancien1996} and \cite[Theorem~III.1]{Johnson1972}:

\begin{prop}\label{gillesres}
Let $\beta <\omega_1$ and let $E$ be a Banach space such that $\mszlenk{E}>\beta$.

\nr{i} There is a separable closed subspace $F$ of $E$ such that $\mszlenk{F} >\beta$.

\nr{ii} If $E\dual$ is norm separable, then for every $\delta >0$ there is a closed subspace $F$ of $E$ such that $\mszlenk{E/F}>\beta$ and $E/F$ has a shrinking basis with basis constant not exceeding $1+\delta$.
\end{prop}

With the exception of the basis constant assertion of part (ii), Proposition~\ref{gillesres} is  proved in \cite{Lancien1996}. Lancien's proof follows closely the proof of \cite[Theorem~III.1]{Johnson1972}, and the extra assertion above regarding the basis constant is easily added to Lancien's result using the observations regarding basis constants in the proof of \cite[Theorem~III.1]{Johnson1972}.

Proposition~\ref{shrinkex} is an immediate consequence of the following
\begin{prop}
Let $\alpha >0$ be a countable ordinal and $E$ a Banach space with $\mszlenk{E} =\omega^\alpha$. Then there exist closed subspaces $F\subseteq E$ and $G\subseteq \ell_2 (F)$ such that $\ell_2(F)/G$ has a shrinking basis and $\mszlenk{\ell_2(F)/G}=\omega^\alpha$.
\end{prop}

\begin{proof}
For each $n\in\nat$, Proposition~\ref{gillesres}(i) yields a separable closed subspace $D_n$ of $E$ such that $\mszlenk{D_n} > \meeszlenk{1/n}{E}$. Let $F = \overline{{\rm span}} \left( \bigcup_{n\in \nat}D_n \right)$. Then
\[
\omega^\alpha = \mszlenk{E} = \sup_n \meeszlenk{1/n}{E} \leqslant \sup_n \mszlenk{D_n} \leqslant \mszlenk{F}\leqslant \mszlenk{E} = \omega^\alpha ,
\]
hence equality holds throughout. In particular, $\mszlenk{F} = \omega^\alpha$ and, as $F$ is a separable Asplund space (indeed, $\mszlenk{F}<\infty$), $F\dual$ is norm separable. For each $n\in\nat$ let $F_n = F$. Then, by Proposition~\ref{gillesres}(ii), for each $n\in\nat$ there is a closed subspace $G_n$ of $F_n$ such that $\mszlenk{F_n /G_n} > \meeszlenk{1/n}{E}$ and $F_n /G_n$ has a shrinking basis with basis constant not exceeding $2$. Let $G$ denote the image of $(\bigoplus_{n\in\nat}G_n)_2$ under its natural embedding into $(\bigoplus_{n\in\nat}F_n)_2$. Then $(\bigoplus_{n\in\nat}F_n)_2 /G$ is naturally isometrically isomorphic to $(\bigoplus_{n\in\nat}F_n/G_n)_2$. Note that $(\bigoplus_{n\in\nat}F_n/G_n)_2$ has a shrinking basis since it is the $\ell_2$-direct sum of a countable family of Banach spaces with shrinking bases that have uniformly bounded basis constants. On the one hand, by Theorem~\ref{trivcase} we have
\[
\textstyle \mszlenk{(\bigoplus_{n\in\nat}F_n)_2/G} \leqslant \mszlenk{(\bigoplus_{n\in\nat}F_n)_2} = \mszlenk{F} = \omega^\alpha .
\]
On the other hand,
\[
\textstyle \mszlenk{(\bigoplus_{n\in\nat}F_n)_2/G} = \mszlenk{(\bigoplus_{n\in\nat}F_n/G_n)_2} \geqslant  \sup_n \meeszlenk{1/n}{E_n} = \mszlenk{E} = \omega^\alpha.
\]
Thus $(\bigoplus_{n\in\nat}F_n)_2/G$ has a shrinking basis and Szlenk index $\omega^\alpha$.
\end{proof}

\begin{prop}\label{weexist}
Let $\alpha$ be an ordinal. Then there exists a Banach space of Szlenk index $\omega^{\alpha+1}$.
\end{prop}

\begin{proof}
Our proof is based on the construction of Szlenk in \cite{Szlenk1968}, by which we construct Banach spaces $E_\beta$ indexed by the class of ordinals $\beta$. Let $E_0 = \set{0}$, $E_{\beta +1} = E_\beta \oplus_1 \ell_2 $ and, if $\beta$ is a limit ordinal, $E_\beta = (\bigoplus_{\gamma <\beta}E_\gamma )_2$. It is shown in \cite[Theorem~4]{Lancien2006} that for this construction we have $\meeszlenk{1}{E_\beta} >\beta$ for all ordinals $\beta$. As the assertion of the proposition is known to be true for $\alpha = 0$ (for example, $\mszlenk{\ell_2} = \omega$), we assume that $\alpha >0$ and let $\beta'$ denote the least ordinal such that $\mszlenk{E_{\beta '}}> \omega^{\alpha}$. Then, by Proposition~\ref{collection}(iii), $\mszlenk{E_{\beta '}}\geqslant \omega^{\alpha+1}$. By Proposition~\ref{collection}(v) and the definition of $\beta'$, it must be that $\beta'$ is a limit ordinal, hence $E_{\beta'}=(\bigoplus_{\beta '' <\beta '}E_{\beta''})_2$. It follows that $\mszlenk{E_{\beta '}} = \mszlenk{(\bigoplus_{\beta '' <\beta '}E_{\beta''})_2} \leqslant \omega^{\alpha+1}$, where the final inequality here follows from Theorem~\ref{accunibound} and the fact that, for all $\eps > 0$,
\[
\sup\{\meeszlenk{\eps}{E_{\beta''}}\mid \beta'' <\beta'\} \leqslant \sup\{\mszlenk{E_{\beta''}}\mid \beta'' <\beta'\} \leqslant \omega^\alpha <\omega^{\alpha +1}.
\]
It is now clear that $\mszlenk{E_{\beta '}} = \omega^{\alpha +1}$, so we are done.
\end{proof}

Implicit in the proof of Proposition~\ref{weexist} is the following fact: for a set $\Lambda$, Banach spaces $\set{E_\lambda \mid \lambda \in \Lambda}$, $p=0$ or $1<p<\infty$ and $\alpha$ an ordinal satisfying $\sup\{\mszlenk{E_\lambda}\mid \lambda\in\Lambda\}\leqslant \omega^\alpha$, we have $\mszlenk{(\bigoplus_{\lambda\in\Lambda}E_\lambda)_p}\leqslant \omega^{\alpha+1}$. This follows easily from Theorem~\ref{accunibound}, but seems to have been known for some time. For example, the separable case was established in \cite[Proposition~15]{Odell2007}, and the result is also implicit in the proof of \cite[Proposition~5]{Lancien2006}.

Proposition~\ref{weexist} and Proposition~\ref{shrinkex} concern themselves with the existence of Banach spaces having a particular Szlenk index. The author is not aware of a complete classification of the possible values of the Szlenk index of a Banach space. Proposition~\ref{collection}(iii) asserts that the Szlenk index of a Banach space is a power of $\omega$. On the other hand, as the Szlenk index of a Banach space $E$ is the supremum of the countable set $\set{\meeszlenk{1/n}{E}\mid n\in\nat}$, it follows that the Szlenk index of a Banach space is of countable cofinality. In particular, if $\alpha$ is an ordinal of uncountable cofinality, then $\alpha$ is a limit ordinal and $\omega^\alpha$ cannot be the Szlenk index a Banach space since $\cf{\omega^\alpha} = \cf{\alpha}\geqslant \omega_1$. In view of this fact and Proposition~\ref{weexist}, a complete classification of values of the Szlenk index of Banach spaces will be achieved if one establishes an affirmative answer to the following question, which we believe to be open:
\begin{question}\label{valueset}
Let $\alpha$ be an ordinal with $\cf{\alpha} = \omega$. Does there exist a Banach space with Szlenk index equal to $\omega^\alpha$?
\end{question}

A partial answer to Question~\ref{valueset} is found in \cite{Odell2007} where it is shown that if $\mathcal{T}_{\omega^\alpha}$ denotes the $\omega^\alpha$th {\cyril} space, where $\alpha <\omega_1$, then $\mszlenk{\mathcal{T}_{\omega^\alpha}} = \omega^{\omega^{\alpha+1}}$. The values taken by the Szlenk index on the class of all operators between Banach spaces will be determined in Proposition~\ref{copequiv} below.

%\subsection{Applications to the theory of $\alpha$-Szlenk operators}\label{faircon}

To conclude the current section, we now apply Proposition~\ref{weexist} to obtain, amongst other things, a characterization of those limit ordinals $\alpha$ for which the operator ideal $\bigcup_{\beta<\alpha}\szlenkop{\alpha}$ is closed.
\begin{prop}\label{copequiv}
Let $\alpha>0$ be an ordinal. The following are equivalent:

\nr{i} $\cf{\alpha}\geqslant \omega_1$.

\nr{ii} $\omega^\alpha$ is not the Szlenk index of any operator between Banach spaces.

\nr{iii} $\szlenkop{\alpha} = \bigcup_{\beta < \alpha}\szlenkop{\beta}$.

\nr{iv} $\alpha$ is a limit ordinal and $\bigcup_{\beta < \alpha}\szlenkop{\beta}$ is closed.
\end{prop}

\begin{proof} We will show that (i)$\Rightarrow$(ii)$\Rightarrow$(iii)$\Rightarrow$(iv)$\Rightarrow$(i).

To see that (i)$\Rightarrow$(ii), suppose that there exists an operator $T$ such that $\omega^\alpha = \mszlenk{T} = \sup\{\meeszlenk{1/n}{T}\mid n\in\nat\}$. Then $\cf{\alpha}\leqslant \cf{\omega^\alpha} =\omega<\omega_1$.

The implication (ii)$\Rightarrow$(iii) is immediate from Proposition~\ref{collection}(iii).

Now suppose that (iii) holds. Then $\bigcup_{\beta < \alpha}\szlenkop{\beta}$ is closed by Theorem~\ref{idealthm}. Moreover, $\alpha$ is a limit ordinal. Indeed, otherwise we may write $\alpha = \zeta +1$, where $\zeta$ is an ordinal, and by Proposition~\ref{weexist} there exists a Banach space $E$ such that $I_E \in \szlenkop{\zeta+1}\setminus \szlenkop{\zeta} = \szlenkop{\alpha}\setminus \bigcup_{\beta < \alpha}\szlenkop{\beta} = \emptyset$, which is absurd.

Finally, we show that (iv)$\Rightarrow$(i). Suppose by way of a contraposition that $\cf{\alpha} = \omega$ and let $\set{\alpha_n \mid n<\omega}\subseteq \alpha$ be cofinal in $\alpha$. Then $\set{\alpha_n +1 \mid n<\omega}$ is also cofinal in $\alpha$, and $\bigcup_{n<\omega} \szlenkop{\alpha_n +1} = \bigcup_{\beta <\alpha}\szlenkop{\beta}$. So to complete the proof, it suffices to construct an operator $T \in \overline{\bigcup_{n<\omega} \szlenkop{\alpha_n +1}} \setminus \bigcup_{n<\omega} \szlenkop{\alpha_n +1}$. To this end, for each $n<\omega$ let $E_n$ be a Banach space whose Szlenk index is $\omega^{\alpha_n+1}$ (c.f. Proposition~\ref{weexist}), and set $E = (\bigoplus_{n<\omega}E_n)_{2}$. Define $T \in \allop (E)$ by setting $T(x_n)_{n<\omega} = ((n+1)^{-1}x_n)_{n<\omega}$ for each $(x_n)_{n<\omega} \in E$. Since $T$ factors $I_{E_n}$ for each $n<\omega$, we have \[\mszlenk{T} \geqslant \sup \set{\mszlenk{E_n}\mid n<\omega} = \sup\set{\omega^{\alpha_n+1}\mid n<\omega} = \omega^\alpha\, ,\] hence $T\notin \bigcup_{n<\omega} \szlenkop{\alpha_n +1}$. On the other hand, with $A_m$ ($m<\omega$) denoting the operator on $E$ that sends $(x_n)_{n<\omega} \in E$ to the element $(y_n)_{n<\omega}$ of $E$ that satisfies $y_n = x_n$ if $n\leqslant m$, and $y_n =0$ otherwise, we have that $I_{E_1 \oplus \ldots \oplus E_m}$ factors $A_mT$ for all $m<\omega$, hence \[ \mszlenk{A_mT} \leqslant \mszlenk{E_1 \oplus \ldots \oplus E_m} = \max \set{\omega^{\alpha_i +1}\mid 1\leqslant i \leqslant m}
. \] In particular, $A_mT \in \bigcup_{n<\omega} \szlenkop{\alpha_n +1}$ for $m<\omega$. As $\lim_{m\rightarrow \omega}\norm{A_mT - T} =0$, it follows that $T \in \overline{\bigcup_{n<\omega} \szlenkop{\alpha_n +1}(E)}$.
\end{proof}

\begin{remark}
The existence of an operator of Szlenk index $\omega^\alpha$ whenever $\cf{\alpha}\leqslant \omega$ (Proposition~\ref{copequiv}(ii)$\Rightarrow$(i)) is used in the proof of \cite[Theorem~5.1]{Brookerc}, where it is shown that if $\beta$ is an ordinal with $\cf{\beta}\leqslant \omega$, then $\szlenkop{\omega^\beta}$ lacks have the factorization property.
\end{remark}

\section{Proof of Lemma~\ref{frount}}\label{techsect}
Our goal in this section is to prove Lemma~\ref{frount}. We proceed via a sequence of lemmas, whose general theme is to establish upper bounds (in terms of set containment) on various derived sets $s_\eps^\alpha (K)$, where $K$ is $w\dual$-compact, $\alpha$ is an ordinal and $\eps >0$. The sets $K$ that we shall consider are typically direct products, for it will be seen later that the set $B_q(K_i \mid 1\leqslant i \leqslant n)$ in the statement of Lemma~\ref{frount} can be `approximated' from above (with respect to set containment) in a convenient way by a finite union of direct products of $w\dual$-compact sets. Indeed, this so-called approximation of $B_q(K_i \mid 1\leqslant i \leqslant n)$ plays a key role in our proof. 

We mention another important aspect of our results in this section. As noted earlier, Lemma~\ref{frount} is used to establish the implication (ii)$\Rightarrow$(iii) of Proposition~\ref{unibound}. Note that in the statement of Proposition~\ref{unibound}(iii), there is no (finite) upper bound on the cardinality of the finite sets $\eff \in \Lambda\fin$. It is thus important for us in this section, when aiming for estimates of $\eps$-Szlenk indices of direct products, to obtain estimates that are independent of the (finite) number of factors in a given direct product. Our efforts in this regard are reflected in the fact that the numbers $M$ and $n$ in the statement of Lemma~\ref{frount} are independent of one another.

We first establish the following general result regarding the behaviour of $s_\eps^\alpha$ derivatives of finite unions of $w\dual$-compact sets.

\begin{lemma}\label{unionlemma}
Let $E$ be a Banach space, $K_1, \ldots , K_n \subseteq E\dual$ $w\dual$-compact sets and $\eps >0$. Let $\alpha$ be an ordinal and $m< \omega$. Then
\begin{itemize}
\item[{\rm (i)}]$\mdset{\alpha}{\eps}{\bigcup_{i=1}^nK_i} \subseteq \bigcup_{i=1}^n \mdset{\alpha}{\eps/2}{K_i}$.
\item[{\rm (ii)}]$\mdset{mn}{\eps}{\bigcup_{i=1}^nK_i} \subseteq \bigcup_{i=1}^n \mdset{m}{\eps}{K_i}$.
\item[{\rm (iii)}] If $\alpha$ is a limit ordinal, then $\mdset{\alpha}{\eps}{\bigcup_{i=1}^nK_i} \subseteq \bigcup_{i=1}^n \mdset{\alpha}{\eps}{K_i}$.
\end{itemize}
\end{lemma}

\begin{proof}
(i) holds trivially for $\alpha = 0$. Suppose that $\beta$ is an ordinal such that (i) holds for all $\alpha \leqslant \beta$ and let $x\in E\dual \setminus \bigcup_{i=1}^ns^{\beta+1}_{\eps/2}(K_i)$. Then for $1\leqslant i\leqslant n$ there is $w\dual$-open $U_i\ni x$ such that $\diam (U_i\cap s^\beta_{\eps/2}(K_i))\leqslant \eps/2$. It follows that for $x_1,\, x_2 \in (\bigcap_{i=1}^nU_i) \cap (s^\beta_\eps(\bigcup_{i=1}^n K_i))$ we have \[ \norm{x_1 -x_2}\leqslant \norm{x_1 - x} +\norm{x-x_2}\leqslant \frac{\eps}{2} + \frac{\eps}{2}= \eps \, ,\] hence $\diam ((\bigcap_{i=1}^nU_i) \cap (s^\beta_\eps (\bigcup_{i=1}^n K_i)))\leqslant \eps$. In particular, $x\notin \mdset{\beta+1}{\eps}{\bigcup_{i=1}^nK_i}$, and so (i) passes to successor ordinals. 

Suppose that $\beta$ is a limit ordinal such that (i) holds for all $\alpha<\beta$. Then
\begin{equation}\label{tingleep}
s^{\beta}_{\eps}\hspace{-1mm}\left({\bigcup_{i=1}^n K_i}\right) = \bigcap_{\alpha<\beta}s^{\alpha}_{\eps}\hspace{-1mm}\left( {\bigcup_{i=1}^nK_i}\right) \subseteq \bigcap_{\alpha <\beta}\, \bigcup_{i=1}^n\mdset{\alpha}{\eps /2}{K_i} .
\end{equation}
Let $x\in \mdset{\beta}{\eps}{\bigcup_{i=1}^n K_i}$. Then for each $\alpha < \beta$ we may choose $i_\alpha \in \set{1, \ldots , n}$ such that $x\in \mdset{\alpha}{\eps/2}{K_{i_\alpha}}$, and for some $i' \in \set{1, \ldots , n}$ the set $\set{\alpha < \beta \mid i_\alpha = i'}$ is cofinal in $\beta$. Hence
\begin{equation}\label{peepyting}
x \in \bigcap_{i_\alpha = i'}\mdset{\alpha}{\eps /2}{K_{i'}} = \bigcap_{\alpha<\beta}\mdset{\alpha}{\eps/2}{K_{i'}} = \mdset{\beta}{\eps/2}{K_{i'}} \subseteq \bigcup_{i=1}^n\mdset{\beta}{\eps /2}{K_i} .
\end{equation}
Since $x\in \mdset{\beta}{\eps}{\bigcup_{i=1}^n K_i}$ was arbitrary, (i) passes to limit ordinals, and thus holds for all ordinals $\alpha$.

Statement (ii) is trivial for $m=0$. To see that it is true for $m=1$, we first let $\mathbb{P}_k = \set{\eff\subseteq \set{1, \ldots , n} \mid \abs{\eff} = k}$, $k \in \nat$. It suffices to show that for all $l < \omega$,
\begin{equation}\label{wagjatch}
s_\eps^l \hspace{-1mm} \left( \bigcup_{i=1}^n K_i \right) \subseteq \left( \bigcup_{i=1}^n \mdset{\mbox{}}{\eps}{K_i} \right) \cup \Bigg( \bigcup_{\eff \in \mathbb{P}_{l+1}} \bigcap_{i \in\eff} K_i \Bigg) ,
\end{equation}
Indeed, taking $l=n$ in (\ref{wagjatch}) gives (ii) with $m=1$ (since $\bigcup_{\eff \in \mathbb{P}_{l+1}} \bigcap_{i \in\eff} K_i  =\emptyset$ when $l=n$). It is clear that (\ref{wagjatch}) holds for $l=0$. Suppose now $l' <\omega$ is such that (\ref{wagjatch}) holds for $l = l'$; we show that it holds also for $l = l'+1$. Let
\begin{equation*}
x \in E\dual \Bigg\backslash \left( \Bigg( \bigcup_{i=1}^n \mdset{\mbox{}}{\eps}{K_i} \Bigg) \cup \Bigg( \bigcup_{\gee \in \mathbb{P}_{l'+2}} \bigcap_{j \in \gee} K_j \Bigg) \right) .
\end{equation*}
We want to show that $x\notin s^{l'+1}_\eps (\bigcup_{i=1}^nK_i)$, so by the induction hypothesis it suffices to assume that
\[
x\in s^{l'}_\eps \hspace{-1mm}\left( \bigcup_{i=1}^nK_i\right) \subseteq \left( \bigcup_{i=1}^n \mdset{\mbox{}}{\eps}{K_i} \right) \cup \Bigg( \bigcup_{\eff \in \mathbb{P}_{l'+1}}\bigcap_{i \in\eff} K_i \Bigg) \, ,
\]
hence
\begin{equation}\label{tiptoptap}
x\in \Bigg( \bigcup_{\eff \in \mathbb{P}_{l'+1}} \bigcap_{i \in\eff} K_i \Bigg) \Bigg\backslash \left( \Bigg( \bigcup_{i=1}^n \mdset{\mbox{}}{\eps}{K_i} \Bigg) \cup \Bigg( \bigcup_{\gee \in \mathbb{P}_{l'+2}} \bigcap_{j \in \gee} K_j \Bigg) \right)\, .
\end{equation}
By (\ref{tiptoptap}) there is (a unique) $\eff_x \in \mathbb{P}_{l'+1}$ such that $x \in \left( \bigcap_{i \in \eff_x}K_i \right)\setminus \left( \bigcup_{i' \notin \eff_x}K_{i'} \right)$. For each $i \in \eff_x$ let $U_i \ni x$ be $w\dual$-open and such that $\diam (U_i \cap K_i ) \leqslant \eps$ and $U_i \cap \bigcup_{i' \notin \eff_x}K_{i'} = \emptyset$. Then $U = \bigcap_{i \in \eff_x}U_i$ is a $w\dual$-neighbourhood of $x$ and
\begin{eqnarray*}
U \cap \left( \Bigg( \bigcup_{i=1}^n \mdset{\mbox{}}{\eps}{K_i} \Bigg) \cup \Bigg( \bigcup_{\eff \in \mathbb{P}_{l'+1}} \bigcap_{i \in \eff} K_i \Bigg) \right) &=& U \cap \bigcap_{i \in \eff_x}K_i = \bigcap_{i \in \eff_x} U_i \cap K_i
\end{eqnarray*}
has norm diameter not exceeding $\eps$ (because $\diam (U_i \cap K_i)\leqslant \eps$ for $i \in \mathcal{F}_x$). It follows then by (\ref{wagjatch}) and the induction hypothesis on $l = l'$ that 
\[
x \notin s_\eps \hspace{-0.5mm} \left( \Bigg( \bigcup_{i=1}^n \mdset{\mbox{}}{\eps}{K_i} \Bigg) \cup \Bigg( \bigcup_{\eff \in \mathbb{P}_{l'+1}} \bigcap_{i \in\eff} K_i \Bigg) \right)  \supseteq s_\eps^{l'+1} \hspace{-0.5mm} \Bigg( \bigcup_{i=1}^n K_i \Bigg) ,
\]
as required. In particular, (\ref{wagjatch}) holds for all $l<\omega$ and (ii) holds for $m=1$.

Suppose $h<\omega$ is such that (ii) holds for all $m\leqslant h$. Then
\[
s^{(h+1)n}_{\eps}\hspace{-1mm} \left( \bigcup_{i=1}^nK_i \right) \subseteq s_\eps^n \hspace{-1mm} \left(\bigcup_{i=1}^n \mdset{h}{\eps}{K_i} \right) \subseteq \bigcup_{i=1}^n \mdset{h+1}{\eps}{K_i},
\]
so that (ii) holds for $m=h+1$, and thus for all $m$ by induction.

For (iii), we prove the case $n=2$, with the general case then following from this case and a straightforward induction on $n$. So we want to show that if $\alpha$ is a nonzero limit ordinal, then
\begin{equation}\label{fudget}
\mdset{\alpha}{\eps}{K_1 \cup K_2} \subseteq \mdset{\alpha}{\eps}{K_1 }\cup \mdset{\alpha}{\eps}{ K_2} .
\end{equation}
To this end, it suffices to consider the case $\alpha = \omega^\beta$, $\beta >0$, since the general case follows from finitely many iterations of this case. Indeed, every limit ordinal $\alpha$ is the sum of finitely many ordinals of the form $\omega^\beta$, $\beta >0$. We proceed by induction on $\beta$. For $\beta = 1$ we note that, by (ii),
\begin{equation}\label{sandwich}
\mdset{\omega}{\eps}{K_1 \cup K_2} = \bigcap_{m< \omega} \mdset{2m}{\eps}{K_1 \cup K_2}\subseteq  \bigcap_{m< \omega} \left( \mdset{m}{\eps}{K_1} \cup \mdset{m}{\eps}{K_2}\right) ,
\end{equation}
and then a similar argument to that used to obtain (\ref{peepyting}) from (\ref{tingleep}) yields (iii) for $\alpha = \omega$.
Suppose now that (\ref{fudget}) holds for $\alpha = \omega^\beta$, some $\beta >0$. Then a straightforward induction on $l < \omega$ shows that for all such $l$ we have
\begin{equation}\label{chewypooey}
\mdset{\omega^{\beta}\cdot l}{\eps}{K_1 \cup K_2} \subseteq \mdset{\omega^{\beta}\cdot l}{\eps}{K_1 }\cup \mdset{\omega^{\beta}\cdot l}{\eps}{ K_2} \, .
\end{equation}
(\ref{chewypooey}) and an argument similar to that used to obtain (\ref{peepyting}) from (\ref{tingleep}) yields
\[
\mdset{\omega^{\beta+1}}{\eps}{K_1 \cup K_2} \subseteq \mdset{\omega^{\beta+1} }{\eps}{K_1} \cup \mdset{\omega^{\beta+1}}{\eps}{K_2}\, ;
\]
in particular, (iii) passes to successor ordinals. The straightforward proof that (iii) passes to limit ordinals uses, once again, a similar cofinality argument to that used to obtain (\ref{peepyting}) from (\ref{tingleep}) above.\end{proof}

The next three lemmas are specifically concerned with $s_\eps^\alpha$ derivatives of direct products of $w\dual$-compact sets, considered as $w\dual$-compact subsets of dual spaces of direct sums of Banach spaces. 

We require more notation. Given Banach spaces $E_1, \ldots , E_n$, nonempty $w\dual$-compact sets $K_1\dual \subseteq E_1\dual, \ldots , K_n\dual \subseteq E_n\dual$, $1\leqslant q<\infty$ and $a_1, \ldots , a_n \geqslant 0$ real numbers such that $\sum_{i=1}^na_i^q \leqslant 1$, for each $\eps > 0$ we define
\[
A_\eps := \set{(\eps_i)_{i=1}^n \in \real^n \, \left| \, \sum_{i=1}^n a_i^q\eps_i^q \geqslant \eps^q \mbox{ and } 0\leqslant \eps_i \leqslant \diam (K_i), \, 1\leqslant i \leqslant n \right. } .
\] In all places where we use the notation $A_\eps$, the $w\dual$-compact sets $K_1, \ldots , K_n$, real numbers $a_1, \ldots , a_n$ and $1\leqslant q <\infty$ will be fixed, so no ambiguity should arise from this notation. It is elementary to see that $A_\eps = \emptyset$ if and only if $\eps^q > \sum_{i=1}^n[a_i \cdot \diam (K_i)]^q$. 

We adopt the notational convention that $s^\alpha_0 (K) = K$ for every ordinal $\alpha$ and $w\dual$-compact $K$.

\begin{lemma}\label{techlem1}
Let $E_1, \ldots , E_n$ be Banach spaces and $K_1 \subseteq E_1\dual , \ldots , K_n\subseteq E_n\dual$ $w\dual$-compact sets. Let $1\leqslant q<\infty$, $\eps >0$ and let $a_1, \ldots , a_n \geqslant 0$ be real numbers such that $\sum_{i=1}^na_i^q \leqslant 1$. Let $p$ be predual to $q$ and consider $\prod_{i=1}^na_iK_i$ as a subset of $(\bigoplus_{i=1}^nE_i)_{p}\dual$. Then, for every $\delta \in (0, \, \eps)$,
\[ 
s_\eps \hspace{-0.5mm} \Bigg( \prod_{i=1}^n a_i K_i \Bigg) \subseteq \bigcup_{(\eps_i) \in A_\delta} \,\,\prod_{i=1}^n a_i s_{\eps_i} (K_i)\, .
\]
\end{lemma}

\begin{proof}
We first suppose that $\eps^q > \sum_{i=1}^n[a_i \cdot \diam (K_i)]^q$. Then $s_\eps ( \prod_{i=1}^n a_i K_i )$ is empty since $\diam (\prod_{i=1}^n a_i K_i)<\eps$. The assertion of the lemma follows.

Suppose now that $\eps^q \leqslant \sum_{i=1}^n[a_i \cdot \diam (K_i)]^q$, so that $A_{\eps'}\neq \emptyset$ for $0<\eps'\leqslant \eps$. Let $\delta \in (0, \, \eps)$, $(a_ix_i)_{i=1}^n \in \mdset{\mbox{}}{\eps}{\prod_{i=1}^na_iK_i}$ and, for $1\leqslant i\leqslant n$, define \[
\delta_i := \inf\set{\diam (K_i \cap U_i)\mid U_i \mbox{ a }w\dual\mbox{-neighbourhood of }x_i}.\] Then $\sum_{i=1}^na_i^q\delta_i^q\geqslant \eps^q >\delta^q $. Let $f: \set{1, \ldots , n} \longrightarrow \real$ be a map such that $\sum_{i=1}^na_i^qf(i)^q\geqslant \delta^q$ and $f(i)\in \set{0}\cup (0, \, \delta_i)$ for all $i$ (note that $[0, \, \delta_i)$ is empty whenever $\delta_i=0$). We \emph{claim} that with $f$ so defined, $x_i\in s_{f(i)}(K_i)$ for $1\leqslant i \leqslant n$. Indeed, if $\delta_i=0$ then $f(i) =0$, hence $x_i \in K_i =s_{f(i)}(K_i)$ by convention. On the other hand, if $\delta_i >0$, then for all $w\dual$-open $U_i \ni x_i$ we have $\diam (K_i \cap U_i) \geqslant \delta_i >f(i) $, hence $x_i \in s_{f(i)}(K_i)$ in this case too. Note that $(f(i))_{i=1}^n\in A_\delta$ since $f(i) \leqslant \delta_i \leqslant \diam (K_i)$ for all $i$ and $\sum_{i=1}^na_i^qf(i)^q\geqslant \delta^q$, hence
\[
(a_ix_i)_{i=1}^n \in \prod_{i=1}^na_i\mdset{\mbox{}}{f(i)}{K_i} \subseteq \bigcup_{(\eps_i)\in A_\delta} \,\,\prod_{i=1}^n a_i s_{\eps_i} (K_i)\, . \qedhere
\]
\end{proof}

\begin{lemma}\label{techlem2}
Let $E_1, \ldots , E_n$ be Banach spaces and $K_1 \subseteq E_1\dual , \ldots , K_n\subseteq E_n\dual$ $w\dual$-compact sets. Let $1\leqslant q<\infty$, $\eps >0$ and let $a_1, \ldots , a_n \geqslant 0$ be real numbers such that $\sum_{i=1}^na_i^q \leqslant 1$. Let $p$ be predual to $q$ and consider $\prod_{i=1}^na_iK_i$ as a subset of $(\bigoplus_{i=1}^nE_i)_{p}\dual$. Then, for every $\delta \in (0, \, \eps)$, $0<m<\omega$ and ordinal $\alpha$,
\begin{equation}\label{fairypoo}
s_\eps^{\omega^\alpha \cdot m} \hspace{-0.5mm} \Bigg( \prod_{i=1}^n a_i K_i \Bigg) \subseteq \bigcup_{(\eps_{i, \, 1}), \ldots , (\eps_{i, \, m}) \in A_{\delta/2}} \,\,\prod_{i=1}^n a_i s_{\eps_{i, \, m}}^{\omega^\alpha}(s_{\eps_{i, \, m-1}}^{\omega^\alpha}(\ldots s_{\eps_{i, \, 1}}^{\omega^\alpha} (K_i)\ldots ))\, .
\end{equation}
\end{lemma}

\begin{proof}
If $\eps^q > \sum_{i=1}^n[a_i \cdot \diam (K_i)]^q$, then $s^{\omega^\alpha\cdot m}_\eps ( \prod_{i=1}^n a_i K_i )$ is empty since $\diam (\prod_{i=1}^n a_i K_i)<\eps$ and $\omega^\alpha \cdot m \geqslant 1$. The assertion of the lemma follows.

Suppose now that $\eps^q \leqslant \sum_{i=1}^n[a_i \cdot \diam (K_i)]^q$, so that $A_{\eps'}\neq \emptyset$ whenever $0<\eps'\leqslant \eps$. For $\alpha =0 $ and $m=1$, (\ref{fairypoo}) is a consequence of Lemma~\ref{techlem1}. Suppose that $\alpha$ is an ordinal such that (\ref{fairypoo}) holds for $m=1, 2, \ldots , k$, for some $0<k<\omega$. We will show that (\ref{fairypoo}) holds for $\alpha$ and $m=k+1$. Fix $\delta \in (0, \, \eps)$ and note that $A_{(\eps +\delta)/4}\subseteq A_{\delta /2}$ since $\delta /2 <(\eps +\delta)/4$. We now detail a method that assigns to each $(\eps_i)_{i=1}^n\in A_{(\eps +\delta)/4}$ an element $(\overline{\eps}_i)_{i=1}^n$ of a certain finite subset of $A_{\delta/2}$. For $(\eps_i)_{i=1}^n\in A_{(\eps +\delta)/4}$ and $1\leqslant i\leqslant n$, define
\[
j_i :=\max \set{j\in\nat\cup\set{0}\mid j(\eps -\delta )\leqslant 4\eps_i}
\]
and set $\overline{\eps}_i = j_i (\eps - \delta)/4$. Note that $\overline{\eps}_i\leqslant \eps_i \leqslant \diam (K_i)$ and
\begin{align*}
\bigg(\sum_{i=1}^na_i^q\overline{\eps}_i^q\bigg)^{1/q} & \geqslant \bigg(\sum_{i=1}^na_i^q\eps_i^q\bigg)^{1/q} - \bigg(\sum_{i=1}^na_i^q(\eps_i -\overline{\eps}_i)^q\bigg)^{1/q} \\ &\geqslant \frac{\eps +\delta}{4} - \frac{\eps -\delta}{4} \\ &= \frac{\delta}{2}\,,
\end{align*}
hence $(\overline{\eps}_i)_{i=1}^n\in A_{\delta/2}$. Moreover, for $(\eps_{i,1})_{i=1}^n, \ldots , (\eps_{i,m})_{i=1}^n \in A_{(\eps +\delta)/4}$ we have
\begin{equation}\label{partytonight}
\mdset{\omega^\alpha}{\eps_{i,  m}}{\mdset{\omega^\alpha}{\eps_{i, m-1}}{\ldots \mdset{\omega^\alpha}{\eps_{i,1}}{K_i}\ldots}} \subseteq \mdset{\omega^\alpha}{\overline{\eps}_{i,  m}}{\mdset{\omega^\alpha}{\overline{\eps}_{i, m-1}}{\ldots \mdset{\omega^\alpha}{\overline{\eps}_{i, 1}}{K_i}\ldots}}
\end{equation}
for all $1\leqslant i\leqslant n$. Let $A= \set{(\overline{\eps}_i)_{i=1}^n \mid (\eps_i)_{i=1}^n \in A_{(\eps +\delta)/4}}\subseteq A_{\delta /2}$. Then $A$ is finite, with
\[
\abs{A} \leqslant \left\lceil \frac{4\cdot \max_{1\leqslant i \leqslant n}\diam (K_i)}{\eps-\delta} +1 \right\rceil^n .
\]
The finiteness of $A$ will allow us to invoke Lemma~\ref{unionlemma} in the next step of our proof. To complete our demonstration that (\ref{fairypoo}) holds for $m=k+1$, we henceforth treat the cases $\alpha =0$ and $\alpha >0$ separately. 

If $\alpha =0$, then for $\delta \in (0, \, \eps)$ we have, by the induction hypothesis, (\ref{partytonight}), Lemma~\ref{unionlemma}(i) and Lemma~\ref{techlem1},
\begin{align*}
s^{k+1}_{\eps}&\hspace{-0.5mm}\Bigg( \prod_{i=1}^na_i K_i\Bigg)  \\ &\subseteq s_\eps \hspace{-1mm}\left( \, \overline{\bigcup_{(\eps_{i,1}),\ldots , (\eps_{i,k})\in A_{(\eps+\delta)/4}}\prod_{i=1}^na_i\mdset{\mbox{}}{\eps_{i,  k}}{\mdset{\mbox{}}{\eps_{i, k-1}}{\ldots \mdset{\mbox{}}{\eps_{i,1}}{K_i}\ldots}}}^{w\dual}\, \right) \\ &\subseteq s_\eps \hspace{-0.5mm}\Bigg( \bigcup_{(\eps_{i,1}),\ldots , (\eps_{i,k})\in A_{(\eps+\delta)/4}}\prod_{i=1}^na_i\mdset{\mbox{}}{\overline{\eps}_{i,  k}}{\mdset{\mbox{}}{\overline{\eps}_{i, k-1}}{\ldots \mdset{\mbox{}}{\overline{\eps}_{i,1}}{K_i}\ldots}} \Bigg) \\ &\subseteq \bigcup_{(\eps_{i,1}),\ldots , (\eps_{i,k})\in A_{(\eps+\delta)/4}}s_{\eps /2} \hspace{-0.5mm}\Bigg(\prod_{i=1}^na_i\mdset{\mbox{}}{\overline{\eps}_{i,  k}}{\mdset{\mbox{}}{\overline{\eps}_{i, k-1}}{\ldots \mdset{\mbox{}}{\overline{\eps}_{i,1}}{K_i}\ldots}} \Bigg) \\ &\subseteq \bigcup_{(\eps_{i,1}),\ldots , (\eps_{i,k}), (\eps_{i,k+1})\in A_{\delta /2}}\, \prod_{i=1}^na_i\mdset{\mbox{}}{\eps_{i,  k+1}}{\mdset{\mbox{}}{\eps_{i, k}}{\ldots \mdset{\mbox{}}{\eps_{i,1}}{K_i}\ldots}} ,
\end{align*}
as required. 

On the other hand, if $\alpha >0$ then it follows from the induction hypothesis, (\ref{partytonight}) and Lemma~\ref{unionlemma}(iii) that
\begin{align*}
s^{\omega^\alpha \cdot (k+1)}_{\eps}&\hspace{-0.5mm}\Bigg( \prod_{i=1}^na_i K_i\Bigg)  \\ &\subseteq s_\eps^{\omega^\alpha}\hspace{-1mm} \left( \, \overline{\bigcup_{(\eps_{i,1}),\ldots , (\eps_{i,k})\in A_{(\eps+\delta)/4}}\prod_{i=1}^na_i\mdset{\omega^\alpha}{\eps_{i,  k}}{\mdset{\omega^\alpha}{\eps_{i, k-1}}{\ldots \mdset{\omega^\alpha}{\eps_{i,1}}{K_i}\ldots}}}^{w\dual}\, \right) \\ &\subseteq s_\eps^{\omega^\alpha}\hspace{-0.5mm} \Bigg( \bigcup_{(\eps_{i,1}),\ldots , (\eps_{i,k})\in A_{(\eps+\delta)/4}}\prod_{i=1}^na_i\mdset{\omega^\alpha}{\overline{\eps}_{i,  k}}{\mdset{\omega^\alpha}{\overline{\eps}_{i, k-1}}{\ldots \mdset{\omega^\alpha}{\overline{\eps}_{i,1}}{K_i}\ldots}} \Bigg) \\ &\subseteq \bigcup_{(\eps_{i,1}),\ldots , (\eps_{i,k})\in A_{(\eps+\delta)/4}}s_{\eps}^{\omega^\alpha} \hspace{-0.5mm}\Bigg(\prod_{i=1}^na_i\mdset{\omega^\alpha}{\overline{\eps}_{i,  k}}{\mdset{\omega^\alpha}{\overline{\eps}_{i, k-1}}{\ldots \mdset{\omega^\alpha}{\overline{\eps}_{i,1}}{K_i}\ldots}} \Bigg) \\ &\subseteq \bigcup_{(\eps_{i,1}),\ldots , (\eps_{i,k}), (\eps_{i,k+1})\in A_{\delta /2}}\, \prod_{i=1}^na_i\mdset{\omega^\alpha}{\eps_{i,  k+1}}{\mdset{\omega^\alpha}{\eps_{i, k}}{\ldots \mdset{\omega^\alpha}{\eps_{i,1}}{K_i}\ldots}} ,
\end{align*}
as we would like.

Finally, suppose that $\beta$ is a nonzero ordinal (either limit or successor) such that (\ref{fairypoo}) holds for all $m<\omega$ and $\alpha<\beta$; we show that (\ref{fairypoo}) then holds for $m=1$ and $\alpha = \beta$. Fix $\delta \in (0, \, \eps)$ and let $A$ be defined as above. Then, since $A\subseteq A_{\delta /2}$, to complete the induction it suffices to show that
\begin{equation}\label{paris}
s^{\omega^\beta}_{\eps}\hspace{-0.5mm}\Bigg({\prod_{i=1}^na_iK_i}\Bigg) \subseteq \bigcup_{(\overline{\eps}_i)\in A}\, \prod_{i=1}^na_i s^{\omega^\beta}_{\overline{\eps}_i}(K_i) \, .
\end{equation}
To prove (\ref{paris}), we shall establish the following two inclusions:
\begin{equation}\label{threet}
s^{\omega^\beta}_{\eps}\hspace{-0.5mm}\Bigg({\prod_{i=1}^na_iK_i}\Bigg) \subseteq \bigcap_{(l,\,\alpha)\in (0,\,\omega)\times \beta} \, \bigcup_{(\overline{\eps}_i)\in A}\, \prod_{i=1}^na_i s^{\omega^\alpha \cdot l}_{\overline{\eps}_i}(K_i)
\end{equation}
and
\begin{equation}\label{libel}
\bigcap_{(l,\,\alpha)\in (0,\,\omega)\times \beta} \, \bigcup_{(\overline{\eps}_i)\in A}\, \prod_{i=1}^na_i s^{\omega^\alpha \cdot l}_{\overline{\eps}_i}(K_i) \subseteq \bigcup_{(\overline{\eps}_i)\in A}\, \prod_{i=1}^na_i s^{\omega^\beta}_{\overline{\eps}_i}(K_i) \, .
\end{equation}

We first deal with (\ref{threet}). To this end, let
\[
x\in s^{\omega^\beta}_{\eps}\hspace{-0.5mm}\Bigg({\prod_{i=1}^na_iK_i}\Bigg) = \bigcap_{(m, \, \alpha)\in (0, \, \omega)\times \beta} s^{\omega^\alpha \cdot m}_\eps \hspace{-0.5mm}\Bigg({\prod_{i=1}^na_iK_i}\Bigg) .
\]
Then, since $\frac{\eps +\delta}{2}<\eps$, it follows from the induction hypothesis and (\ref{partytonight}) that

\begin{align*}
x&\in \bigcap_{(m, \, \alpha)\in (0, \, \omega)\times \beta}\, \bigcup_{(\eps_{i, \, 1}), \ldots , (\eps_{i, \, m})\in A_{(\eps+\delta)/4}}\, \, \prod_{i=1}^na_i\mdset{\omega^\alpha}{\eps_{i,  m}}{\mdset{\omega^\alpha}{\eps_{i, m-1}}{\ldots \mdset{\omega^\alpha}{\eps_{i,1}}{K_i}\ldots}} \\
&\subseteq \bigcap_{(m, \, \alpha)\in (0, \, \omega)\times \beta}\, \bigcup_{(\overline{\eps}_{i, \, 1}), \ldots , (\overline{\eps}_{i, \, m})\in A}\, \, \prod_{i=1}^na_i\mdset{\omega^\alpha}{\overline{\eps}_{i,  m}}{\mdset{\omega^\alpha}{\overline{\eps}_{i, m-1}}{\ldots \mdset{\omega^\alpha}{\overline{\eps}_{i,1}}{K_i}\ldots}}\, .
\end{align*}
So for each $(m, \, \alpha ) \in (0, \, \omega )\times \beta$ there are $(\overline{\eps}_{i, \, 1,\, m,\, \alpha})_{i=1}^n, \ldots , (\overline{\eps}_{i, \, m,\, m,\, \alpha})_{i=1}^n \in A$ such that
\begin{equation}\label{hillsteer}
x\in \prod_{i=1}^na_i\mdset{\omega^\alpha}{\overline{\eps}_{i, \, m,\, m,\, \alpha}}{\mdset{\omega^\alpha}{\overline{\eps}_{i, \, m-1,\, m,\, \alpha}}{\ldots \mdset{\omega^\alpha}{\overline{\eps}_{i, \, 1,\, m,\, \alpha}}{K_i}\ldots}}\, .
\end{equation}

Suppose $l\in (0, \, \omega)$ and $\alpha <\beta$ and set $m_l = \abs{A}\cdot l$. Then there is a subset $J_{l, \, \alpha} \subseteq \set{1, \, 2, \ldots , m_l}$ with $\abs{J_{l,\, \alpha}}=l$ and $\abs{\set{(\overline{\eps}_{i, \, j,\, m,\, \alpha})_{i=1}^n\mid j\in J_{l,\, \alpha}}}=1$. Let $(\overline{\eps}_{i, \, l,\, \alpha})_{i=1}^n $ denote the unique element of $\set{(\overline{\eps}_{i, \, j,\, m,\, \alpha})_{i=1}^n\mid j\in J_{l,\, \alpha}} (\subseteq A)$. We may write $J_{l,\,\alpha} = \set{j_1< j_2 < \ldots < j_l}$, and then by (\ref{hillsteer}) we have, in particular,
\begin{align}
x&\in \prod_{i=1}^na_i\mdset{\omega^\alpha}{\overline{\eps}_{i, \, m_l,\, m_l,\, \alpha}}{\mdset{\omega^\alpha}{\overline{\eps}_{i, \, m_l-1,\, m_l,\, \alpha}}{\ldots \mdset{\omega^\alpha}{\overline{\eps}_{i, \, 1,\, m_l,\, \alpha}}{K_i}\ldots}}\notag \\
&\subseteq \prod_{i=1}^na_i\mdset{\omega^\alpha}{\overline{\eps}_{i, \, j_l,\, m_l,\, \alpha}}{\mdset{\omega^\alpha}{\overline{\eps}_{i, \, j_{l-1},\, m_l,\, \alpha}}{\ldots \mdset{\omega^\alpha}{\overline{\eps}_{i, \, j_1,\, m_l,\, \alpha}}{K_i}\ldots}} \notag \\
&= \prod_{i=1}^na_i\mdset{\omega^\alpha\cdot l}{\overline{\eps}_{i, \, l,\, \alpha}}{K_i}\notag \\
&\subseteq \bigcup_{(\overline{\eps}_i)\in A}\, \prod_{i=1}^na_i\mdset{\omega^\alpha\cdot l}{\overline{\eps}_{i}}{K_i}\, .\notag
\end{align}
As $l\in(0, \, \omega)$ and $\alpha<\beta$ were arbitrary, (\ref{threet}) follows.

We now prove (\ref{libel}). Let
\[
y\in \bigcap_{(l,\,\alpha)\in (0,\,\omega)\times \beta} \, \bigcup_{(\overline{\eps}_i)\in A}\, \prod_{i=1}^na_i s^{\omega^\alpha \cdot l}_{\overline{\eps}_i}(K_i),
\]
and for each $l\in (0, \, \omega)$ and $\alpha<\beta$ let $(\overline{\eps}_{i, \, (l, \, \alpha)})_{i=1}^n \in A$ be such that
\[
y\in \prod_{i=1}^na_i s^{\omega^\alpha \cdot l}_{\overline{\eps}_{i, \, (l, \, \alpha)}}(K_i).
\]
For each $(\overline{\eps}_i)_{i=1}^n \in A$, let \[ \mathcal{A}[(\overline{\eps}_i)_{i=1}^n] = \set{\omega^\alpha \cdot l \mid 0<l<\omega, \, \alpha<\beta, \, (\overline{\eps}_{i, \, (l, \, \alpha)})_{i=1}^n=(\overline{\eps}_i)_{i=1}^n} .\] Since $\set{\omega^\alpha \cdot l \mid 0<l<\omega, \, \alpha<\beta}$ is cofinal in $\omega^\beta$ and $\set{\mathcal{A}[(\overline{\eps}_i)_{i=1}^n]\mid (\overline{\eps}_i)_{i=1}^n\in A}$ is a finite partition of $\set{\omega^\alpha \cdot l \mid 0<l<\omega, \, \alpha<\beta}$, there exists $(\overline{\rho}_i)_{i=1}^n\in A$ such that $\mathcal{A}[(\overline{\rho}_i)_{i=1}^n]$ is cofinal in $\omega^\beta$. It follows that
\begin{align*} 
y \in \bigcap_{\xi \in \mathcal{A}[(\overline{\rho}_i)_{i=1}^n]}\, \prod_{i=1}^n a_i s^\xi_{\overline{\rho}_i}(K_i)&\subseteq \prod_{i=1}^n a_i \Bigg( \bigcap_{\xi \in \mathcal{A}[(\overline{\rho}_i)_{i=1}^n]}\, s^\xi_{\overline{\rho}_i}(K_i)\Bigg) \\ &=  \prod_{i=1}^n a_i \Bigg( \bigcap_{\xi <\omega^\beta}\, s^\xi_{\overline{\rho}_i}(K_i)\Bigg) \\&= \prod_{i=1}^n a_i s^{\omega^\beta}_{\overline{\rho}_i}(K_i)\\ &\subseteq \bigcup_{(\overline{\eps}_i) \in A}\, \prod_{i=1}^na_is^{\omega^\beta}_{\overline{\eps}_i}(K_i) \, .
\end{align*}
At last, the proof of Lemma~\ref{techlem2} is complete.
\end{proof}

\begin{lemma}\label{techlema}
\baselineskip=17pt 
Let $E_1, \ldots , E_n$ be Banach spaces and $K_1 \subseteq E_1\dual  , \ldots , K_n\subseteq E_n\dual$ nonempty $w\dual$-compact sets. Let $1\leqslant q<\infty$, $\eps >0$ and let $a_1, \ldots , a_n \geqslant 0$ be real numbers such that $\sum_{i=1}^na_i^q \leqslant 1$. Let $\displaystyle d= \max_{1\leqslant i \leqslant n}\diam (K_i) $ and let $m, \, M\in \nat$ be such that $M\geqslant m\geqslant 2$ and $(2^q-1)\eps^q M\geqslant 8^q d^q (m-1)$. Let $p$ be predual to $q$ and consider $\prod_{i=1}^na_iK_i$ as a subset of $(\bigoplus_{i=1}^nE_i)_{p}\dual$. If $\alpha$ is an ordinal such that $\mdset{\omega^\alpha \cdot m}{\eps/8}{K_i} =\emptyset$
for all $1\leqslant i \leqslant n$, then $s_\eps^{\omega^\alpha \cdot M}\hspace{-1.5mm}\left( \prod_{i=1}^n a_iK_i \right) =\emptyset  $.
\end{lemma}

\begin{proof}
If $\eps^q > \sum_{i=1}^n[a_i \cdot \diam (K_i)]^q$, then $s^{\omega^\alpha\cdot M}_\eps ( \prod_{i=1}^n a_i K_i )$ is empty since $\diam (\prod_{i=1}^n a_i K_i)<\eps$ and $\omega^\alpha \cdot M \geqslant 1$. The assertion of the lemma follows.

So suppose now that $\eps^q \leqslant \sum_{i=1}^n[a_i \cdot \diam (K_i)]^q$. Then $A_{\eps'}\neq \emptyset$ whenever $0<\eps'\leqslant \eps$. Applying Lemma~\ref{techlem2} with $\delta = \eps /2$, we see that $s_\eps^{\omega^\alpha \cdot M}\left( \prod_{i=1}^n a_iK_i \right)$ is contained in a union of sets of the form
\begin{equation}\label{durryfump}
\prod_{i=1}^na_i\mdset{\omega^\alpha}{\eps_{i,  M}}{\mdset{\omega^\alpha}{\eps_{i, M-1}}{\ldots \mdset{\omega^\alpha}{\eps_{i,1}}{K_i}\ldots}} \, ,
\end{equation}
where $(\eps_{i,1})_{i=1}^n, \, (\eps_{i,2})_{i=1}^n,\ldots , (\eps_{i,M})_{i=1}^n\in A_{\eps/4}$. For each such product (\ref{durryfump}),
\[
a_1^q\Bigg( \sum_{j=1}^{M}\eps^q_{1, j} \Bigg) + a_2^q\Bigg( \sum_{j=1}^{M}\eps^q_{2, j}\Bigg) + \ldots + a_n^q\Bigg( \sum_{j=1}^{M}\eps^q_{n, j}\Bigg) \geqslant \frac{M\eps^q}{4^q} \, .
\]
Since $\sum_{i=1}^na_i^q\leqslant 1$, there is $h\in \set{1, \ldots , n}$ such that $\sum_{j=1}^{M}\eps^q_{h, j} \geqslant M \eps^q /4^q$. We claim that at least one of the following two conditions holds for such $h$:
\begin{itemize}
\item[(a)] There exists a subset $\set{j_1 <j_2<\ldots <j_m}\subseteq \set{1,\, 2, \ldots , M}$ such that $\min \set{\eps_{h, \, j_1}, \ldots , \eps_{h, \, j_m}}\geqslant \eps/8$.
\item[(b)] There exists $j\leqslant M$ such that $\eps_{h, \, j}>d$.
\end{itemize}
Indeed, suppose that (a) does not hold. Then there are distinct $j_1, \ldots , j_{m-1}$ in $ \set{1, \ldots , M}$ such that $\eps_{h,\, j} <\eps /8$ whenever $j \in \set{1, \ldots , M}\setminus \set{j_1, \ldots , j_{m-1}}$. It follows then that
\begin{align*}
\sum_{k=1}^{m-1}\eps_{h, \, j_k}^{q} &> \frac{M\eps^q}{4^q} - (M - m+1)\left( \frac{\eps}{8} \right)^q \\ &= M \Big( \left(\frac{\eps}{4} \right)^q - \left( \frac{\eps}{8} \right)^q \Big) +(m-1)\left(\frac{\eps}{8}\right)^q \\ &> M \Big( \left(\frac{\eps}{4} \right)^q - \left( \frac{\eps}{8} \right)^q \Big) \\ &\geqslant d^q(m-1) \, .
\end{align*}
Thus $\eps_{h, \, j_k}^{q} >d^q$ for some $k\leqslant m-1$, hence $\eps_{h, \, j_k} >d$ for some $k\leqslant m-1$. In particular, (b) holds whenever (a) does not. 

If (b) holds, then the factor $a_h\mdset{\omega^\alpha}{\eps_{h,  M}}{\mdset{\omega^\alpha}{\eps_{h, M-1}}{\ldots \mdset{\omega^\alpha}{\eps_{h,1}}{K_h}\ldots}}$ is empty since $\diam (K_h) \leqslant d < \eps_{h,\, j}$ for $j$ satisfying (b). It follows then that the product $\prod_{i=1}^na_i\mdset{\omega^\alpha}{\eps_{i,  M}}{\mdset{\omega^\alpha}{\eps_{i, M-1}}{\ldots \mdset{\omega^\alpha}{\eps_{i,1}}{K_i}\ldots}}$ is empty also, giving the desired result. On the other hand, if (a) holds then
\begin{align*}
\mdset{\omega^\alpha}{\eps_{h,  M}}{\mdset{\omega^\alpha}{\eps_{h, M-1}}{\ldots \mdset{\omega^\alpha}{\eps_{h,1}}{K_h}\ldots}} &\subseteq \mdset{\omega^\alpha}{\eps_{h,  j_m}}{\mdset{\omega^\alpha}{\eps_{h, j_{m-1}}}{\ldots \mdset{\omega^\alpha}{\eps_{h,j_1}}{K_h}\ldots}} \\ 
&\subseteq \mdset{\omega^\alpha \cdot m}{\eps /8}{K_h} .
\end{align*}
We conclude that $s_\eps^{\omega^\alpha \cdot M}\left( \prod_{i=1}^n a_iK_i \right)$ is contained in a union of direct products of the form (\ref{durryfump}), with each such direct product having a factor contained in a scalar multiple of one of the sets $\mdset{\omega^\alpha \cdot m}{\eps /8}{K_i}$, $1\leqslant i \leqslant n$. From this it is clear that if $\mdset{\omega^\alpha \cdot m}{\eps /8}{K_i}=\emptyset$ for all $1\leqslant i \leqslant n$, then $s_\eps^{\omega^\alpha \cdot M}\left( \prod_{i=1}^n a_iK_i \right)\subseteq \emptyset$.
\end{proof}

The next and final lemma required for our proof of Lemma~\ref{frount} shows how we can put a set $B_q(K_i\mid 1\leqslant i \leqslant n)$ inside a finite union of direct products of $w\dual$-compact sets in a way that will be useful for us.

\begin{lemma}\label{lecondsast}
Let $E_1, \ldots , E_n$ be Banach spaces, $K_1\subseteq E_1\dual, \ldots , K_n\subseteq E_n\dual$ nonempty, absolutely convex, $w\dual$-compact sets, $1\leqslant q<\infty$ and $l\in\nat$. Let $L=\nat^n\cap (l+n^{1/q})\cball{\ell_q^n}$. Then\[  B_q(K_i\mid 1\leqslant i \leqslant n) \subseteq \bigcup_{(k_i)_{i=1}^n\in L} \prod_{i=1}^n \frac{k_i}{l}K_i \, . \]
\end{lemma}

\begin{proof}
Let $(a_i)_{i=1}^n \in \cball{\ell_q^n}$ and set $j_i = \inf \{ j\in\nat\mid l\abs{a_i}<j \}$, $1\leqslant i \leqslant n$. Then $j_i -1 \leqslant l\abs{a_i}$ for all $i$, hence $\norm{(j_i)_{i=1}^n}_{\ell_q^n} \leqslant \norm{(la_i)_{i=1}^n}_{\ell_q^n} + {n^{1/q}} \leqslant l+n^{1/q}$. In particular, $(j_i)_{i=1}^n \in L$. As the sets $K_i$, $1\leqslant i \leqslant n$, are absolutely convex, we have $a_i K_i \subseteq \frac{j_i}{l}K_i$ for all $i$, hence $\prod_{i=1}^n a_i K_i \subseteq \prod_{i=1}^n \frac{j_i}{l}K_i$. It follows that
\[
\quad B_q(K_i\mid 1\leqslant i \leqslant n) = \bigcup_{(a_i) \in \cball{\ell_q^n}}{\, \prod_{i=1}^n a_i K_i} \subseteq \bigcup_{(k_i)_{i=1}^n\in L}\prod_{i=1}^n \frac{k_i}{l}K_i \, . \qedhere
\]
\end{proof}

We note a few points of interest regarding the sets $\bigcup_{(k_i)\in L} \prod_{i=1}^n \frac{k_i}{l}K_i$ from Lemma~\ref{lecondsast}. For each $l\in\nat$, let $L_l=\nat^n\cap (l+n^{1/q})\cball{\ell_q^n}$. Then the intersection of the collection $\{\bigcup_{(k_i)\in L_l} \prod_{i=1}^n \frac{k_i}{l}K_i\}_{l\in\nat}$ is precisely $B_q(K_i\mid 1\leqslant i \leqslant n)$; this follows from the observation that for $l\in\nat$, each point of $\bigcup_{(k_i)\in L_l} \prod_{i=1}^n \frac{k_i}{l}K_i$ is no greater than $n^{1/q}\cdot l^{-1}\cdot \max \{ \diam (K_i )\mid 1\leqslant i \leqslant n\}$ in norm distance from $B_q(K_i\mid 1\leqslant i \leqslant n)$. We may thus think of $\{\bigcup_{(k_i)\in L_l} \prod_{i=1}^n \frac{k_i}{l}K_i\}_{l\in\nat}$ as a sequence of increasingly closer approximations to the set $B_q(K_i\mid 1\leqslant i \leqslant n)$, and our need to closely approximate $B_q(K_i\mid 1\leqslant i \leqslant n)$ is reflected by our choice of $l$ in the following proof of Lemma~\ref{frount}.

\begin{altproof4}
Fix $\delta \in (0, \, \eps /16)$. Let $l \geqslant 16\delta n^{1/q}(\eps -16\delta)^{-1}$ be an integer and let $ L = \nat^n\cap (l+n^{1/q})\cball{\ell_q^n} $. By Lemma~\ref{lecondsast} and the hypothesis of Lemma~\ref{frount},
\[
s^{\omega^\alpha \cdot M}_{\eps}\hspace{-0.5mm}\Bigg(\bigcup_{(k_i)\in L} \prod_{i=1}^n \frac{k_i}{l}K_i \Bigg) \supseteq s^{\omega^\alpha \cdot M}_{\eps}(B_q(K_i\mid 1\leqslant i \leqslant n)) \supsetneq \emptyset\, .
\]
Thus, since $L$ is finite, by Lemma~\ref{unionlemma}(i) there exists $(h_i)_{i=1}^n\in L$ such that
\begin{equation}\label{weednaughter}
s^{\omega^\alpha \cdot M}_{\eps/2}\hspace{-1.5mm}\left(\prod_{i=1}^n \frac{h_i}{l}K_i \right) \neq \emptyset \, .
\end{equation}
Let $\rho = (1+ \frac{n^{1/q}}{l})^{-1}$. By (\ref{weednaughter}) and the homogeneity of the derivations $s^\gamma_{\eps'}$ (where $\gamma$ is an ordinal and $\eps'>0$), we have
\begin{equation}\label{peedpaughter}
s^{\omega^\alpha \cdot M}_{\rho \eps/2}\hspace{-1.5mm}\left(\prod_{i=1}^n \frac{\rho h_i}{l}K_i \right) = \rho s^{\omega^\alpha \cdot M}_{\eps/2}\hspace{-1.5mm}\left(\prod_{i=1}^n \frac{h_i}{l}K_i \right) \neq \emptyset \, .
\end{equation}
Thus, since $ ||(\frac{\rho h_i}{l})_{i=1}^n||_{\ell_q^n} \leqslant 1$, it follows from (\ref{peedpaughter}) and Lemma~\ref{techlema} that there is $i \leqslant n$ such that $\mdset{\omega^\alpha \cdot m}{\rho \eps /16}{K_{i}}\neq \emptyset$. As $\rho\eps /16 \geqslant\delta$, we conclude that $\mdset{\omega^\alpha \cdot m}{\delta}{K_{i}}\supseteq \mdset{\omega^\alpha \cdot m}{\rho\eps /16}{K_{i}}\supsetneq \emptyset$. This completes the proof.
\end{altproof4}

\begin{remark}
Lemma~\ref{frount} is similar to \cite[Lemma~5.9]{Brookerc}. Though many of the arguments and preliminary results used here in the proof of Lemma~\ref{frount} have been employed similarly in the proof of \cite[Lemma~5.9]{Brookerc}, neither of these technical lemmas are strong enough to be used in place of the other in the proofs of the respective theorems for which they have been developed.
\end{remark}

\section*{Acknowledgements}
The author thanks Dr. Rick Loy for many valuable suggestions that have improved the presentation of the results.

Part of the research presented here was completed during a visit of the author to the D\'{e}partement de Math\'{e}matiques at the Universit\'{e} de Franche-Comt\'{e}, and the author thanks the department for their kind hospitality. The author is especially grateful to Prof. Gilles Lancien for many stimulating discussions during his stay.

\vspace{10mm}

\footnotesize
\noindent Philip Brooker

\noindent Mathematical Sciences Institute

\noindent Australian National University

\noindent Canberra ACT 0200, Australia

\vphantom{pahb}

\noindent \textit{E-mail addresses}: philip.brooker@anu.edu.au, philip.a.h.brooker@gmail.com


\begin{thebibliography}{00}

\bibitem{Alspach1979}
D.~E. Alspach and Y.~Benyamini.
\newblock {$C(K)$} quotients of separable {${\cal L}\sb{\infty }$} spaces.
\newblock {Israel J. Math.}, 32(2-3):145--160, 1979.

\bibitem{Bessaga1960}
C.~Bessaga and A.~Pe{\l}czy{\'n}ski.
\newblock Spaces of continuous functions. {IV}. {O}n isomorphical
  classification of spaces of continuous functions.
\newblock {Studia Math.}, 19:53--62, 1960.

\bibitem{Brookerc}
P.~Brooker.
\newblock {A}splund operators and the
  {S}zlenk index.
\newblock {Preprint}. arXiv:1003.5710v3 [math.FA]

\bibitem{Deville1993}
R.~Deville, G.~Godefroy, and V.~Zizler.
\newblock {Smoothness and renormings in {B}anach spaces}, volume~64 of {Pitman Monographs and Surveys in Pure and Applied Mathematics}.
\newblock Longman Scientific \& Technical, Harlow, 1993.

\bibitem{Engelking1989}
R.~Engelking.
\newblock {General topology}, volume~6 of {Sigma Series in Pure
  Mathematics}.
\newblock Heldermann Verlag, Berlin, second edition, 1989.

\bibitem{Fabian2001}
M.~Fabian, P.~Habala, P.~H{\'a}jek, V.~Montesinos~Santaluc{\'{\i}}a, J.~Pelant, and V.~Zizler.
\newblock {Functional analysis and infinite-dimensional geometry}.
\newblock CMS Books in Mathematics/Ouvrages de Math\'ematiques de la SMC, 8.
  Springer-Verlag, New York, 2001.

\bibitem{Giles1982}
J.~Giles.
\newblock {Convex analysis with application in the differentiation of
  convex functions}, volume~58 of {Research Notes in Mathematics}.
\newblock Pitman (Advanced Publishing Program), Boston, Mass., 1982.

\bibitem{H'ajek2007}
P.~H{\'a}jek and G.~Lancien.
\newblock Various slicing indices on {B}anach spaces.
\newblock {Mediterr. J. Math.}, 4(2):179--190, 2007.

\bibitem{H'ajek2007a}
P.~H{\'a}jek, G.~Lancien, and V.~Montesinos.
\newblock Universality of {A}splund spaces.
\newblock {Proc. Amer. Math. Soc.}, 135(7):2031--2035 (electronic), 2007.

\bibitem{H'ajek2008}
P.~H{\'a}jek, V.~ Montesinos~Santaluc{\'{\i}}a, J.~Vanderwerff, and
  V.~Zizler.
\newblock {Biorthogonal systems in {B}anach spaces}.
\newblock CMS Books in Mathematics/Ouvrages de Math\'ematiques de la SMC, 26.
  Springer, New York, 2008.

\bibitem{Heinrich1980}
S.~Heinrich.
\newblock Closed operator ideals and interpolation.
\newblock {J. Funct. Anal.}, 35(3):397--411, 1980.

\bibitem{Johnson1972}
{\sc Johnson, W.~B., and Rosenthal, H.~P.}
\newblock On {$\omega \sp{\ast} $}-basic sequences and their applications to
  the study of {B}anach spaces.
\newblock {\em Studia Math. 43\/} (1972), 77--92.

\bibitem{Lancien1993}
G.~Lancien.
\newblock Dentability indices and locally uniformly convex renormings.
\newblock {Rocky Mountain J. Math.}, 23(2):635--647, 1993.

\bibitem{Lancien1996}
G.~Lancien.
\newblock On the {S}zlenk index and the weak{$\sp *$}-dentability index.
\newblock {Quart. J. Math. Oxford Ser. (2)}, 47(185):59--71, 1996.

\bibitem{Lancien2006}
G.~Lancien.
\newblock A survey on the {S}zlenk index and some of its applications.
\newblock {RACSAM Rev. R. Acad. Cienc. Exactas F\'\i s. Nat. Ser. A Mat.},
  100(1-2):209--235, 2006.

\bibitem{Odell2007}
E.~Odell, Th. Schlumprecht, and A.~Zs{\'a}k.
\newblock Banach spaces of bounded {S}zlenk index.
\newblock {Studia Math.}, 183(1):63--97, 2007.

\bibitem{Pelczy'nski1969}
A.~Pe{\l}czy{\'n}ski.
\newblock Universal bases.
\newblock {Studia Math.}, 32:247--268, 1969.

\bibitem{Pietsch1980}
A.~Pietsch.
\newblock {Operator ideals}, volume~20 of {North-Holland Mathematical
  Library}.
\newblock North-Holland Publishing Co., Amsterdam, 1980.
%\newblock Translated from German by the author.

\bibitem{Reuinov1978}
O.~Re{\u\i}nov.
\newblock R{N}-sets in {B}anach spaces.
\newblock {Funktsional. Anal. i Prilozhen.}, 12(1):80--81, 96, 1978.

\bibitem{Rosenthal2003}
H.~Rosenthal.
\newblock The {B}anach spaces {$C(K)$}.
\newblock In {Handbook of the geometry of Banach spaces, Vol.\ 2}, pages
  1547--1602. North-Holland, Amsterdam, 2003.

\bibitem{Samuel1984}
C.~Samuel.
\newblock Indice de {S}zlenk des {$C(K)$} ({$K$} espace topologique compact
  d\'enombrable).
\newblock In {Seminar on the geometry of {B}anach spaces, {V}ol. {I}, {II}
  ({P}aris, 1983)}, volume~18 of {Publ. Math. Univ. Paris VII}, pages
  81--91. Univ. Paris VII, Paris, 1984.

\bibitem{Semadeni1971}
Z.~Semadeni.
\newblock {Banach spaces of continuous functions. {V}ol. {I}}.
\newblock PWN---Polish Scientific Publishers, Warsaw, 1971.
\newblock Monografie Matematyczne, Tom 55.

\bibitem{Stegall1981}
C.~Stegall.
\newblock The {R}adon-{N}ikod\'ym property in conjugate {B}anach spaces. {II}.
\newblock {Trans. Amer. Math. Soc.}, 264(2):507--519, 1981.

\bibitem{Szlenk1968}
W.~Szlenk.
\newblock The non-existence of a separable reflexive {B}anach space universal
  for all separable reflexive {B}anach spaces.
\newblock {Studia Math.}, 30:53--61, 1968.

\end{thebibliography}
\end{document}